\documentclass[11pt]{article}%
\usepackage{amssymb,amsmath,amsfonts,amsthm,array,bm,bbm,color}%
\usepackage{graphicx}
\providecommand{\U}[1]{\protect\rule{.1in}{.1in}}

\usepackage{setspace}
\newcommand{\Keywords}{\par \noindent\textbf{Keywords:}}

\numberwithin{equation}{section}

\numberwithin{equation}{section}
\newtheorem{thm}{Theorem}[section]
\newtheorem{prop}[thm]{Proposition}
\newtheorem{lem}[thm]{Lemma}
\newtheorem{cor}[thm]{Corollary}
\newtheorem{rmk}[thm]{Remark}

\newtheorem*{lem*}{Lemma}
\newtheorem*{thm*}{Theorem}
\newcommand*{\RMN}[1]{\uppercase\expandafter{\romannumeral#1}}

\def\<{\langle}
\def\>{\rangle}
\def\d{{\rm d}}
\def\L{\mathcal{L}}
\def\div{{\rm div \,}}
\def\curl{{\rm curl \,}}
\def\E{\mathbb{E}}
\def\N{\mathbb{N}}
\def\P{\mathbb{P}}
\def\R{\mathbb{R}}
\def\T{\mathbb{T}}

\DeclareMathOperator{\tr}{tr}

\begin{document}

\title{Strong Uniqueness by Kraichnan Transport Noise for the 2D Boussinesq Equations with Zero Viscosity}
\author{Shuaijie Jiao\footnotemark[1] \quad Dejun Luo\footnotemark[2]}

	\maketitle
	
	\vspace{-20pt}
	\renewcommand{\thefootnote}{\fnsymbol{footnote}}
	\footnotetext[1]{Email: jiaoshuaijie@amss.ac.cn. School of Mathematical Sciences, University of Chinese Academy of Sciences, Beijing 100049, China, and Academy of Mathematics and Systems Science, Chinese Academy of Sciences, Beijing 100190, China}
	
	\footnotetext[2]{Email: luodj@amss.ac.cn. SKLMS, Academy of Mathematics and Systems Science, Chinese Academy of Sciences, Beijing 100190, China, and School of Mathematical Sciences, University of Chinese Academy of Sciences, Beijing 100049, China}

\begin{abstract}
	We investigate the inviscid 2D Boussinesq equations driven by rough transport noise of Kraichnan type with regularity index $\alpha\in (0,1/2)$. For all $1<p<\infty$, we establish the existence and  uniqueness of probabilistic strong solutions for all $L^p$ initial vorticity and $L^2$ initial temperature, under the parameter constraint $0<\alpha< 1-1/(p\wedge 2)$. The key ingredient is the anomalous regularity due to the noise proven by Coghi and Maurelli \cite{CogMau} who dealt with stochastic 2D Euler equations. Combining techniques from analysis and probability, we demonstrate how the additional regularity from noise compensates the singularity due to the nonlinear parts and coupled terms.
\end{abstract}


\section{Introduction}

This paper is concerned with the vorticity form of stochastic 2D Boussinesq equation perturbed by transport noise of Kraichnan type on $[0,T]\times \R^2$, namely
\begin{equation}\label{Boussinesq}
	\left\{\aligned
	& \partial_t \omega + u\cdot \nabla\omega +\dot{W} \circ \nabla\omega =\partial_1\theta,\\
	& \partial_t\theta + u\cdot \nabla\theta = \Delta \theta, \quad u=K*\omega,
	\endaligned \right.
\end{equation}
where $\omega$ and $\theta$ are respectively the fluid vorticity and density/temperature, $K$ is the Biot-Savart kernel, $\circ$ stands for the Stratonovich multiplication and $W= W(t,x)$ is the famous Kraichnan noise which is white in time, colored and divergence-free in space. More precisely, the Fourier transform of its space covariance matrix $Q(x-y)=Q(x,y):= \E[W(1,x) \otimes W(1,y)]$ is given by
\begin{equation}\label{Q^}
	\widehat{Q}(\xi)=\langle \xi\rangle^{-(2+2\alpha)}\bigg(I_2-\dfrac{\xi\otimes \xi}{\lvert \xi\rvert^2 }\bigg),\quad \xi\in \mathbb{R}^2,
\end{equation}
where $\langle \xi\rangle:=(1+\vert \xi\vert^2)^{1/2},\alpha\in (0,1)$ and $I_2$ is the $2\times 2$ identity matrix. It is known that the noise can be represented as $W(t,x)=\sum_{k}\sigma_k(x)W_t^k$, where $\{\sigma_k\}$ is a family of smooth divergence-free vector fields and $\{W^k\}$ is a sequence of independent Brownian motions, see \cite[Section 2.1]{GalLuo23}.

\subsection{Motivations}

The standard deterministic 2D Boussinesq system for incompressible fluid flows in $\R^2$ reads as
\begin{equation}\label{det-Bouss}
	(B_{\nu,\kappa})~\left\{\aligned
	& \partial_t u + u\cdot \nabla u +\nabla p =\theta {\bf e}_2 +\nu \Delta u,\\
	& \partial_t\theta + u\cdot \nabla\theta = \kappa \Delta\theta, \quad \div u=0,
	\endaligned \right.
\end{equation}
where $u$ denotes the velocity field, $p$ the pressure, $\theta$ the temperature or density of fluid, $\nu$ the kinematic viscosity, $\kappa$ the thermal diffusivity and ${\bf e}_2$ is the unit vector in the vertical direction. The Boussinesq system plays a pivotal role in atmospheric and oceanic sciences; the interested reader can refer to the books \cite{Maj,Ped} for more information.

The solution theory of the viscous 2D Boussinesq system $(B_{\nu,\kappa})$, $\kappa,\nu>0$, shares close similarities with the 2D Navier-Stokes equation, in particular, the global regularity and the uniqueness of Leray-Hopf weak solution are well known, see for instance \cite{CanDi}.
In the inviscid case ($\kappa=\nu=0$), due to the absence of gradient estimates on the temperature field $\theta$, the global regularity of Boussinesq equations $(B_{0,0})$ turns out to be extremely difficult and still remains open, see for instance \cite[Remark 1.3.2]{Wujia}.
For the zero diffusivity case $(B_{\nu,0})$, the global regularity  has been established in \cite{Cha,Hou} for initial data $(u_0,\theta_0)\in  H^m\times H^m, ~m>2$, and was later extended \cite{DanPai08} to unique energy solutions for $L^2$ initial data, see also \cite{BJQW}.
As for the case of $\nu=0,\kappa>0$, the zero viscosity system $(B_{0,\kappa})$  admits  global $H^m$ solutions with initial data $(u_0,\theta_0)\in  H^m\times H^m, ~m>2$ (cf. \cite{Cha}), with subsequent extensions to rougher initial data in \cite{HimKer}.
When it comes to uniqueness of weak solutions of $(B_{0,\kappa})$ with zero viscosity, the results are much less complete. The reason is simple: the inviscid Boussinesq system $(B_{0,\kappa})$ can be regarded as 2D Euler equation coupled with a transport-diffusion equation.
Analogous to the classical Yudovich theory \cite{Yud} for Euler equation, Danchin and Paicu \cite{DanPai09} proved uniqueness of weak solutions to $(B_{0,\kappa})$ for bounded initial vorticity $\omega_0$ and specific regularity conditions on $\theta_0$; see also \cite{PaiZhu} for  related recent results.

The uniqueness of weak solutions to 2D Euler equation
  \begin{equation}\label{2D-Euler}
  \partial_t \omega + u\cdot\nabla \omega= 0
  \end{equation}
with $L^1\cap L^p$ vorticity, $1<p<\infty$, is a long-standing open problem in fluid dynamics.
In recent years, several ``negative'' results appeared.
By adding a carefully designed forcing term to \eqref{2D-Euler}, Vishik \cite{Vis1,Vis2} constructed non-unique solutions with null initial condition.
Bressan et al. \cite{BreMur,BreShe} proposed a (numerically assisted) scheme for demonstrating the existence of non-unique weak solutions to 2D Euler equations.
For $\omega_0\in \dot{H}^{-1}\cap L^1\cap {L^p},\, p>2$, Mengual \cite{Men} proved that there exist infinitely many bounded admissible solutions $u \in C_tL^2$ to the 2D Euler equation without force.
The most recent progress in this direction is the work \cite{BCK}, where Bru\`e et al. developed a new convex integration scheme and constructed, for some $p>1$ and $\omega_0$ belonging to a dense subset of $\dot{H}^{-1}\cap L^p$, infinitely many non-conservative weak solutions $\omega\in C([0,1];\dot{H}^{-1} \cap L^p)$ to the 2D Euler equation on $\T^2$ without external forces.
Although the uniqueness problem has not been completely settled, we tend to believe that the $L^1\cap L^p$ solutions of 2D Euler equation are non-unique; this also holds for Boussinesq system $(B_{0,\kappa})$ since it includes the 2D Euler equation as a special case.

In the last decades, the growing theory of regularization by noise demonstrates that suitable noises can improve the solution theory of many systems.
It was first shown in the seminal work \cite{FGP} that Stratonovich transport noise can improve the well-posedness theory of linear transport equations with H\"older continuous drift vector field.
After that, there are numerous works aimed at extending the results to nonlinear PDEs, cf. \cite{FlaBook, Gess} for surveys of some early results and the introductions of \cite{CogMau, GalLuo23} for more comprehensive references.
Regarding the choice of appropriate noises in fluid dynamics equations, Holm \cite{Holm2015} derived Stratonovich stochastic fluid equations from stochastic variational principles, and the paper \cite{FP22} obtained transport noise from additive noise through multiscale arguments.
Nowadays, it is broadly accepted that the transport noise in Stratonovich form, which models the small-scale fluctuations of fluid motions, is a physically meaningful random perturbation in fluid equations.

Among the vast literature in the area of regularization by noise for fluid dynamics equations, we mention two types of results: one is suppression of blow-up of strong solutions, the other is restoration of uniqueness of weak solutions.
In \cite{FlaLuo21}, Flandoli and Luo have shown that transport noise can suppress the blow-up of analytic strong solutions of 3D Navier-Stokes equation in vorticity form, obtaining global existence with high probability.
This is achieved by extending the It\^o-Stratonovich diffusion limit of transport noise introduced by Galeati \cite{Gal} to nonlinear PDEs, see \cite{FGL21b, FHLN22, Luo23, Agr24a} for some subsequent works.
In the more recent paper \cite{Agr24}, Agresti proved the global existence of strong solutions to the velocity form of stochastic 3D Navier-Stokes equation with transport noise and small hyper-dissipation, by combining the scaling limit argument in \cite{FlaLuo21} and the stochastic maximal regularity theory developed in \cite{NVW} and a series of follow-up articles.

Concerning restoration of uniqueness of weak solutions to fluid equations by suitable transport noise, Coghi and Maurelli \cite{CogMau} have successfully proved that rough Kraichnan noise can restore the uniqueness of the $L^1\cap L^p$ solutions to the vorticity form of 2D Euler equation on $\R^2$, which greatly improves the solution theory in deterministic setting.
The key ingredient of their proof is the emergence of anomalous $L_T^2\dot{H}^{-\alpha}$ regularity in the usual  $C_T\dot{H}^{-1}$ energy estimate, where $\alpha$ is the index in the definition of the Kraichnan noise \eqref{Q^}.
Indeed, if we formally apply It\^o's formula to $\d \|\omega\|_{\dot{H}^{-1}}^2$, terms related to noise can be written as ${\rm tr}[(Q(0)- Q) D^2 G]\ast \omega, \omega\big\>$, where $G$ is the Green function on $\R^2$ and $\tr[\cdot]$ means the trace of matrices.
It was shown in \cite{CogMau} that there exist two positive constants $c_{\alpha},C$ such that
\begin{equation}\label{key-estimate}
	\big\< {\rm tr}[(Q(0)- Q) D^2 G]\ast \omega, \omega\big\> \le -c_{\alpha} \|\omega \|_{H^{-\alpha}}^2 + C \|\omega \|_{\dot H^{-1}}^2,
\end{equation}
where the negative part is of vital importance in proving the pathwise uniqueness since the additional regularity from noise compensates for the singularity of the nonlinear term.
In subsequent papers \cite{JiaoLuo1, BGM1}, this idea was further developed and applied to the modified surface quasi-geostrophic (mSQG) equation.
The anomalous regularity in \eqref{key-estimate} is closely related to the Kraichnan noise which was first introduced in the seminal work \cite{Kraichnan68} to study passive scalars in turbulent flows.
The covariance \eqref{Q^} of rough Gaussian random field mimics the inertial range of turbulence, whose structure has been extensively studied in \cite{JanRai02,JanRai04}, see also \cite[Section 2.3]{CogMau}.
In \cite[Lemma 4.1]{CogMau}, Coghi and Maurelli fully utilized this unique structure to obtain the regularity estimate \eqref{key-estimate}.

The goal of this work is to extend results in \cite{CogMau} to the Boussinesq system.
It seems that the most natural extension is to consider the vorticity form of system $(B_{0,0})$ with Kraichnan noise:
\begin{equation*}
	\left\{\aligned
	& \partial_t \omega + u\cdot \nabla\omega +\dot{W} \circ \nabla\omega =\partial_1\theta,\\
	& \partial_t\theta + u\cdot \nabla\theta + \dot{W}\circ \nabla\theta= 0,
	\endaligned \right.
\end{equation*}
where $W$ is the rough Kraichnan noise defined in \eqref{Q^}. However, as the deterministic system mentioned before, it seems impossible to get the estimate of $\|\nabla\theta\|_{L^p}$ for $p\in (1,\infty)$, in spite of the Kraichnan noise in the $\theta$-equation, thus making the control of $\|\omega\|_{L^p}$ unavailable.
Therefore, we choose to consider the stochastic Boussiensq system $(B_{0,\kappa})$ with $\kappa>0$ (as in \cite{Luo21}), and prove existence and uniqueness of probabilistic strong solutions  of \eqref{Boussinesq} for $L^p$ initial vorticity and suitable initial temperature.
The coupled structure of system \eqref{Boussinesq} causes several difficulties.
The strategy for proving pathwise uniqueness is to estimate the $\dot{H}^{-1}\times \dot{H}^{-1}$ norm of the difference of two solutions with the same initial condition, which is a little surprising since the natural norm for $\theta$ seems to be $L^2$ due to the transport structure.
This may be because the anomalous regularity from Kraichnan noise only appears in Sobolev spaces of negative order, see \cite[Remark 1.2]{GGM}.
While $\|\omega_t \|_{L^p}$ can be controlled by initial vorticity in case of 2D Euler equation, here, however, the $L^\infty_T L^p$ norm of $\omega$ for $p\ne 2$ depends on $\|\omega \|_{C_T \dot H^{-1}}$ in a pathwise manner, forcing us to make additional efforts, see Section 1.3 below for more discussions.

Before finishing this part, let us mention a few examples where noise fails to improve the solution theory of fluid equations.
In \cite{HZZ}, Hofmanov\'a et al. have proved non-uniqueness in law of analytically weak solutions of 3D Navier-Stokes equation driven by additive, linear multiplicative and nonlinear cylindrical noises; in \cite{HLP}, non-uniqueness of probabilistic strong solutions to 3D Euler equations perturbed by smooth Stratonovich transport noise has been proved.
These results also suggest the necessity of choosing rough transport noise to achieve regularization by noise. Finally, in contrast to our results, Yamazaki \cite{Yama} proved non-uniqueness in law for 2D Boussinesq system with fractional viscosity driven by additive and linear multiplicative noise.

\subsection{Main results}

From the explicit formula \eqref{Q^}, we have
\begin{equation*}
	Q(0)=\sum_{k}\sigma_k(x)\otimes\sigma_k(x)=\frac{\pi}{2\alpha}I_2, \quad \forall x\in \R^2.
\end{equation*}
Then some simple computations lead to the It\^o formulation of \eqref{Boussinesq}:
\begin{equation}\label{Boussinesq-Ito}
	\left\{\aligned
	& \partial_t \omega + u\cdot \nabla\omega + \sum_{k}\sigma_k \cdot \nabla\omega \,\dot{W}^k=\partial_1\theta
	+\frac{\pi}{4\alpha}\Delta \omega,\\
	& \partial_t\theta + u\cdot \nabla\theta = \Delta \theta.
	\endaligned \right.
\end{equation}
In the sequel we always work with this formulation. We also remark that for every smooth divergence-free vector field $u$ on $\R^2$ and $\omega=\curl u$, the following identity holds:
\begin{equation*}
	\curl\div (u\otimes u)=u\cdot \nabla \omega.
\end{equation*}
We shall replace the nonlinear part $u\cdot \nabla \omega$ in \eqref{Boussinesq-Ito} by $\curl\div (u\otimes u)$.

\begin{thm}\label{main-thm}
	Given $\big(\Omega, \mathcal F, (\mathcal F_t)_t, \P, (W^k)_k \big)$, where $(\Omega, \mathcal F, (\mathcal F_t)_t, \P)$ is a filtered probability space satisfying the usual conditions and $(W^k)_k$ are a sequence of independent $(\mathcal F_t)_t$-adapted real Brownian motions.
	Let $1<p<\infty$ and $0<\alpha<\frac{1}{2}$ satisfy $0<\alpha<1-\frac{1}{p\wedge2}$. Assume $\omega_0\in \dot{H}^{-1}\cap L^1\cap L^p$ and $\theta_0\in \dot{H}^{-1}\cap L^1\cap L^2$, there exists a pathwise unique pair $(\omega,\theta)$ of $(\mathcal{F}_t)_t$-adapted $\dot{H}^{-1}\times\dot{H}^{-1}$-valued continuous process satisfying the bounds
	\begin{equation}\label{path-bdd}
		\begin{split}
			\P\text{-a.s.}, \quad
			&\sup_{t\in [0,T]} \|\omega_t \|_{L^1\cap L^{p}} \le \|\omega_0 \|_{L^1\cap L^{p}}+\int_{0}^{T}\|\nabla\theta_t\|_{L^1\cap L^{p}} \,\d t,\\
			&\sup_{t\in [0,T]} \|\theta_t \|_{L^{2}}^2+ \int_{0}^{T}\|\nabla\theta_t\|_{L^{2}}^2\,\d t\leq2\|\theta_0\|_{L^2}^2
			,\quad
			\int_{0}^{T}\|\nabla\theta_t\|_{L^1\cap L^{p}} \,\d t<\infty,
		\end{split}
	\end{equation}

    \begin{align}\label{energy-bdd}
    	\E\bigg[\sup_{t\in [0,T]}\|\omega_t \|_{\dot{H}^{-1}}^2 \bigg] + \int_{0}^{T}\E\big[\|\omega_t \|_{\dot{H}^{-\alpha}}^2\big]\, \d s\lesssim_{\alpha,T} \|\omega_0\|_{\dot{H}^{-1}}^2 + \|\theta_0\|_{L^2}^2,
    \end{align}
    and the following equalities hold in distribution sense for every $t\in [0,T]$,
    \begin{equation}\label{weak}
    	\begin{split}
    		\omega_t &= \omega_0
    		-\int_{0}^{t}\curl\div(u_s \otimes u_s)\, \mathrm{d}s
    		-\sum_{k}\int_{0}^{t}\mathrm{div}(\sigma_k\omega_s) \, \d W_s^k
    		+ \int_{0}^{t} \partial_1\theta_s  \, \d s
    		+\frac{\pi}{4\alpha}\int_{0}^{t} \Delta \omega_s \, \d s,\\
    		\theta_t&= \theta_0
    		-\int_{0}^{t}\mathrm{div}(u_s\theta_s)\, \mathrm{d}s
    		+  \int_{0}^{t} \Delta \theta_s \, \d s,
    		\quad u_t=K*\omega_t.
    	\end{split}
    \end{equation}
\end{thm}

\begin{rmk}
	\begin{itemize}
		\item[\rm (1)] As is shown by the estimates \textbf{Cases 1}--\textbf{3} at the end of Section \ref{appr}, the solution that we construct satisfies the following pathwise bound: for every $1<p<\infty$,
		\begin{equation}\label{ex-path-bdd}
			\int_{0}^{T}\|\nabla\theta_t\|_{L^1\cap L^{p}} \,\d t
			\leq C\Big(p,T, \|\omega_0 \|_{L^1\cap L^p}, \|\theta_0\|_{L^1\cap L^2},\|\omega\|_{C_T\dot{H}^{-1}}\Big), \quad \P\text{-a.s.},
		\end{equation}
		where the right-hand side means a $\R_+$-valued random variable depending on $\|\omega\|_{C_T\dot{H}^{-1}}$ linearly and initial data.
		\item[\rm (2)] If $\alpha\in (0,1)$ and $(\omega_0,\theta_0)\in \dot{H}^{-1}\times L^2$, then the existence of weak solutions (in probability sense) fulfilling the energy bound \eqref{energy-bdd} can be shown as in \cite[Section 6]{CogMau}.
		
		\item[\rm (3)] As is shown in the proof below, the conditions in Theorem \ref{main-thm} can be slightly relaxed in the $L^2$-setting. In particular, if $(\omega_0,\theta_0)\in (\dot{H}^{-1}\cap L^2)^2$, then there exists a unique probabilistic strong solution $(\omega, \theta)$ of the equations \eqref{Boussinesq-Ito} in the product space
		\begin{equation*}
			L^{\infty}(\Omega\times [0,T],L^2)  \cap L^2\big(\Omega,C([0,T], \dot{H}^{-1})\big) \times L^{\infty}\big(\Omega,L^2 ([0,T],  H^{1})\big)\cap L^2\big(\Omega,C([0,T],  \dot H^{-1}) \big).
		\end{equation*}
	    When $\theta_0=0$, the Boussinesq system reduces to 2D Euler equation, and the above result removes the condition $\omega_0\in L^1$ in \cite[Theorem 2.12]{CogMau} when $p=2$.
	\end{itemize}
\end{rmk}

\begin{rmk}
The condition $\alpha<1/2$ seems indispensable due to the product estimates of Sobolev functions in Lemma \ref{product}. When $1<p<2$, pathwise uniqueness is ensured by the condition $\alpha<1-1/p$, which can be regarded as the sub-critical regime.
The critical regime $\alpha=1-1/p$ is more complicated as the noise-induced regularity balances exactly the singularity in nonlinear terms.
In the recent paper \cite{BGM1}, Bagnara et al. can deal with this endpoint case in mSQG equation by a decomposition strategy, separating the unknown into two components, one of which is critical but arbitrarily small, while the other is sub-critical.
This strategy also works for our setting, but we do not pursue such extension here since it requires lots of extra effort.
\end{rmk}

\begin{rmk}
If we neglect the coupled term $u\cdot\nabla \theta$, then \eqref{Boussinesq-Ito} becomes
	\begin{equation*}
		\left\{\aligned
		& \partial_t \omega + u\cdot \nabla\omega + \sum_{k}\sigma_k \cdot \nabla\omega \,\dot{W}^k=\partial_1\theta
		+\frac{\pi}{4\alpha}\Delta \omega,\\
		& \partial_t\theta  = \Delta \theta.
		\endaligned \right.
	\end{equation*}
Since $\div u=\div \sigma_k=0$, at least formally, we have for all $1\leq p \leq \infty$,
    \begin{equation*}
    	\|\omega\|_{L_T^{\infty} L^p}\leq
    	 \|\omega_0 \|_{ L^{p}}+\int_{0}^{T}\|\nabla\theta_t\|_{L^{p}} \,\d t.
    \end{equation*}
Hence it suffices to estimate the $L_T^1L^p$ norm of $\nabla\theta$ to obtain the $L^p$-control of $\omega$.
From the characterization of Besov spaces by heat semigroup \cite[Theorem 2.34]{BYD}, we have
    \begin{equation*}
    	\int_{0}^{\infty} \|\nabla\theta_t\|_{L^p}\, \d t
    	\asymp \|\theta_0\|_{\dot{B}_{p,1}^{-1}}.
    \end{equation*}
As a consequence, it seems that the best possible initial condition on temperature is $\theta_0\in \dot{B}_{1,1}^{-1}\cap \dot{B}_{p,1}^{-1}$. However, the presence of transport term $u\cdot\nabla \theta$ makes it rather difficult to achieve this goal. Fortunately, the term $u\cdot\nabla \theta$ vanishes in $L^2$-estimate. In Section 2.2, We take full advantage of this property to obtain bounds on $ \|\nabla \theta\|_{L_T^1L^p}$ for $1\leq p<\infty$.
\end{rmk}

\subsection{Main strategy}

Our approach relies heavily on the role of noise.  More specifically, thanks to the rough Kraichnan noise, the solution enjoys an $\dot{H}^{-\alpha}$ anomalous regularity  when doing $\dot{H}^{-1}$ energy estimate for  vorticity, which is crucial to prove uniqueness.
Our results can be regarded as an extension of the work on the Euler equation by Coghi and Maurelli \cite{CogMau} to system of coupled equations.
Based on a refined estimate obtained in \cite[Lemma 3.1]{JiaoLuo}, we can improve the results in \cite[Theorem 2.12]{CogMau} by removing the unnatural constraint $\alpha>\frac{2}{p\wedge 2}-1$.
On the other hand, in the context of the Boussinesq system, the $L^p$ estimates for vorticity depend on the gradient of temperature and are random variables except in the case $p=2$ (see \textbf{Cases 1}--\textbf{3} at the end of Section \ref{appr}), which is entirely different from the situation of Euler equations.
To overcome this difficulty, we employ the stopping time technique in Sections \ref{Converge} and \ref{proof} to shrink the sample space, thereby obtaining the $L^p$ control of vorticity.

The method of constructing smooth approximate solutions is also different from the work \cite{CogMau}. We choose to mollify the initial data and the noise, adding a small viscosity term $\delta\Delta \omega$ in \eqref{Boussinesq-Ito} to build approximate solutions, while Coghi and Maurelli regularized the initial data, the noise and the nonlinear kernel and used the results on Mckean\,-Vlasov equations to construct approximate solutions in \cite[Section 3]{CogMau}. Our approach simplifies the proof to some extent and is more suitable to handle with the Boussinesq system.

Another difference from the previous work \cite{CogMau} is the proof of existence of solutions.
In \cite{CogMau}, the authors proved the (analytic and probabilistic) weak existence by means of classical compactness methods, while our strategy is to directly prove that the approximate solutions converge in probability in the energy space $C_T\dot{H}^{-1}\times C_T\dot{H}^{-1}$, implying the existence of probabilistic strong solutions.
This strategy essentially reveals the stability of the system \eqref{Boussinesq-Ito} with respect to rough Kraichnan noise and continuous dependence on initial conditions.
In contrast to the compactness method, we obtain strong convergence in $C_T\dot{H}^{-1}$ norm which is global in space variable, a fact crucial for obtaining strong convergence in $L_T^1 \dot{W}^{1,p}$ of the approximate sequence $(\theta^{\delta})_{\delta>0}$ for $1<p<\infty$ in Corollary \ref{cor3}.
However, the error terms related to different (smooth) noises cannot be fully controlled, thus we can only prove that the approximate solutions form a Cauchy sequence in $C_T \dot{H}^{-1}$ when $p\geq2$, see Remark \ref{tech-error-noise}.
As for the case $p\in (1,2)$, we reconstruct the approximate solutions by exploiting the relevant well-posedness results in the case $p\geq2$. The new regularized systems correspond to smooth initial data and the same fixed rough Kraichnan noise,
which is a key point since the above-mentioned error term due to different noises no longer exists in energy estimate.

\subsection{Structure of paper and notations}
In Section \ref{appr}, we review some fundamental properties of the Kraichnan noise, give the definition of the approximate solutions $(\omega^\delta,\theta^{\delta})_{\delta>0}$ and derive some uniform energy estimates on them. In Section \ref{Converge}, we prove that $(\omega^{\delta})$ is a Cauchy sequence in $C([0,T];\dot{H}^{-1})$, which will lead to stronger convergence of $(\theta^{\delta})$ due to the coupled structure of the equations. Section \ref{proof} contains the proof of Theorem \ref{main-thm}. For completeness and consistency, we provide in the appendices sketched  proofs of well-posedness of viscous Boussinesq system (Lemma \ref{wellpoesd-reg-model}) and a technical result (Corollary \ref{cor3}).

Now we collect some frequently used notations. Let $(\Omega,\mathcal{F},(\mathcal{F}_t)_t,\P)$ be a filtered probability space which satisfies the usual conditions and $\mathbb{E}$ denotes the expectation.
For $x = (x^1,x^2)\in \mathbb{R}^2$, we define $x^{\perp} = (x^2,-x^1)$ as the vector obtained by rotating $x$ clockwise 90 degrees and $\langle x \rangle=(1+\lvert x \rvert^2)^{\frac{1}{2}}$. Write the open (resp. closed) ball of center $x$ and radius $R$ as $B(x,R)$ (resp. $\bar{B}(x,R))$. A smooth radial function $\varphi:\R^2 \to [0,1]$ is called a bump function if it is supported in $B(0,1)$ and equals to $1$ on $B(0,1/2)$.

Throughout this paper, the notation $\langle\cdot,\cdot\rangle$ stands both for the scalar product in a Hilbert space and the pairing between a space and its dual. For a tempered distribution $u\in \mathcal{S}', \hat{u}$ denotes its Fourier transform. For $s\in \mathbb{R}$, let $\dot{H}^s(\mathbb{R}^2)\,(\text{resp. }H^s(\mathbb{R}^2))$ denote the usual homogeneous (resp. inhomogeneous) Sobolev spaces, see for instance \cite[Chapter 1]{BYD}.

The symbol $f\lesssim_{\alpha} g$ for two functions $f$ and $g$ means that there exists a positive constant $C_{\alpha}$ depending on the parameter $\alpha$ such that $f(x)\leq C_\alpha g(x)$ for all $x$. We write $f\asymp g$ if $f\lesssim g$ and $g\lesssim f$.

For $p>1,T>0$ and a Banach space $X$, we write $C_TX$ or  $C([0,T];X)$ (resp. $L_T^pX$ or $L^p([0,T];X)$) as the set of $X$-valued continuous functions (resp. $X$-valued $L^p$ integrable functions). Similar notation will be used for  general vector valued Sobolev spaces.

Throughout this paper we shall  always use the Einstein summation convention for the components of vector fields on $\R^2$.

The following estimates about the product of functions in Sobolev spaces will be used frequently, see \cite[Lemma 7.1]{CogMau} or \cite[Corollary 2.55]{BYD} for a detailed proof.
\begin{lem}\label{product}
	Set $a,b\in (-1,1)$ such that $a+b>0$, then for every $f\in \dot{H}^{a},~g\in \dot{H}^{b}$, the product $fg\in \dot{H}^{a+b-1}$ and it holds
	\begin{equation*}
		\|fg\|_{\dot{H}^{a+b-1}}
		\lesssim_{\alpha} \|f\|_{\dot{H}^{a}}\,
		\|g\|_{\dot{H}^{b}}.
	\end{equation*}
\end{lem}

\section{Approximate solutions}\label{appr}

This section is divided into two parts. First, we collect some fundamental properties about the structure of the noise and construct smooth approximate solutions corresponding to smooth initial data and regularized noise. Then, we establish some uniform estimates on the approximate sequence, including the main energy estimate in Proposition \ref{key-bdd} which gives an extra control of the $\dot{H}^{-\alpha}$ norm of the vorticity and is crucial in the proof of pathwise uniqueness.

\subsection{Construction of approximate solutions}

We start with the structure of the rough Kraichnan noise. Let $0<\alpha<1$, recall that the spatial isotropic covariance matrix $Q: \R^2\rightarrow \R^2\times \R^2$ is characterized by

\begin{equation*}
	\widehat{Q}(\xi)=\langle \xi\rangle^{-(2+2\alpha)}\bigg(I_2-\dfrac{\xi\otimes \xi}{\lvert \xi\rvert^2 }\bigg),\quad \xi\in \mathbb{R}^2.
\end{equation*}
The following results are taken from  \cite[Section 2]{CogMau}.

\begin{lem}
	There exists a sequence of divergence free vector fields $\{\sigma_k\}$ which forms a complete orthonormal basis of $H_{\sigma}^{1+\alpha}(\R^2)$, such that
	
	\begin{equation*}
		Q(x-y)=\sum_{k}\sigma_k(x)\otimes\sigma_k(y), \quad \forall x,y\in \R^2
	\end{equation*}
    and
    \begin{equation*}
    	\sup_{x,y\in \R^2}\sum_{k} \vert \sigma_k(x)\vert \, \vert \sigma_k(y)\vert<\infty,
    \end{equation*}
    where the series converges absolutely, uniformly on compact sets.
\end{lem}

Now we mollify the covariance $Q$ as in \cite[Section 3]{CogMau}. For $0<\delta<1$, suppose $\rho^\delta$ is a smooth function such that $\widehat{\rho^\delta}$ is the indicator function of the ball $B(0,1/\delta)$. Take
\begin{equation*}
	\begin{split}
		\sigma_k^{\delta}=\rho^{\delta} \ast \sigma_k,\quad
		Q^{\delta}=\rho^{\delta}*\rho^{\delta} \ast Q.
	\end{split}
\end{equation*}
Note that
\begin{equation}  \label{Q-delta}
	\widehat{Q^{\delta}}(\xi)
	=\widehat{Q}(\xi)\widehat{\rho^\delta}(\xi)^2
	=\langle \xi\rangle^{-(2+2\alpha)}\bigg(I_2-\dfrac{\xi\otimes \xi}{\lvert \xi\rvert^2 }\bigg)
	\widehat{\rho^\delta}(\xi)^2.
\end{equation}
Therefore, $Q^{\delta}\in C^{\infty}(\mathbb{R}^2)$ satisfies $\Vert Q^{\delta} \Vert_{L^{\infty}}
\leq \big\Vert \widehat{Q^{\delta}} \big\Vert_{L^1}
\leq \big\Vert \widehat{Q} \big\Vert_{L^1}$.
Moreover, it is easy to check  $Q^{\delta}(x-y)=\sum_{k}\sigma_k^{\delta}(x)\otimes \sigma_k^{\delta}(y)$. Also, as $\delta\rightarrow0$,
\begin{equation}\label{c-delta}
	Q^{\delta}(0)=\frac{1}{2}\bigg(\int_{\mathbb{R}^2}
	\langle \xi\rangle^{-(2+2\alpha)}\widehat{\rho^\delta}(\xi)^2\,\mathrm{d}\xi\bigg)I_2=:c_{\delta}I_2\rightarrow\dfrac{\pi}{2\alpha}I_2=Q(0), \quad \Big|c_\delta-\frac{\pi}{2\alpha}\Big|\lesssim_{\alpha}\delta^{2\alpha}.
\end{equation}
For the smooth covariance $Q^{\delta}$, the next lemma shows its derivatives are also summable, see \cite[Proposition 2.6]{GalLuo23} or \cite[Lemma 3.1]{CogMau} for a proof.

\begin{lem}\label{reg-Q}
	For every $m,n\in \N$, it holds
	\begin{equation*}
		D^{m+n}Q^{\delta}(x-y)=(-1)^n\sum_{k}D^m\sigma_k^{\delta}(x)\otimes D^n\sigma_k^{\delta}(y)
	\end{equation*}
    and
	\begin{equation*}
			\sup_{x,y\in \R^2}\sum_{k} \vert D_x^m\sigma_k^{\delta}(x)\vert \, \vert D_y^m\sigma_k^{\delta}(y)\vert<\infty,
	\end{equation*}
    where the series converge absolutely, uniformly on compact sets.
\end{lem}

Now we turn to regularize the initial data. Given $\omega_0\in \dot{H}^{-1}\cap L^1\cap L^p$ for some $1< p<\infty$, let $u_0=K*\omega_0$, where $K$ is the Biot-Savart kernel. Take two bump functions $\bar{\rho}\text{ and } \chi$ satisfying $\int \bar{\rho}\,\d x=1$ and for $\delta>0$, we set $\bar{\rho}_\delta=\delta^{-2}\bar{\rho}(\delta^{-1\cdot}),\chi_\delta=\chi(\delta\cdot)$. Define
\begin{equation*}
	\begin{split}
		u_0^{\delta}&=(u_0*\bar{\rho}_\delta)\chi_\delta,\\
		\omega_0^{\delta}&=\curl(u_0^{\delta})=(\omega_0*\bar{\rho}_\delta)\chi_\delta-(u_0*\bar{\rho}_\delta)\cdot \nabla^\perp \chi_\delta,
	\end{split}
\end{equation*}
from \cite[Section 3.2]{CogMau}, it is obvious that the approximate sequence $(\omega_0^\delta)_{\delta}\subset C_c^{\infty}\cap \dot{H}^{-1}$ satisfies
\begin{equation}\label{converge-initial-1}
	\omega_0^{\delta}\rightarrow\omega_0 \quad\text{ in } \dot{H}^{-1}\cap L^1\cap L^p \text{ as } \delta\to 0.
\end{equation}
Similarly, given $\theta_0\in \dot{H}^{-1}\cap L^1\cap L^2$,  the approximate sequence defined by
\begin{equation*}
	\theta_0^\delta=\curl \big((\theta_0*K*\bar{\rho}_\delta)\chi_\delta\big)
\end{equation*}
satisfies
\begin{equation}\label{converge-initial-2}
	\theta_0^\delta\in C_c^{\infty}\cap \dot{H}^{-1}\quad \text{ and }\quad
	\theta_0^{\delta}\rightarrow\theta_0 \quad\text{ in } \dot{H}^{-1}\cap L^1\cap L^2\text{ as } \delta\rightarrow0.
\end{equation}

Throughout the subsequent analysis, we always assume that $(\omega_0^{\delta})_{\delta>0}$ satisfy in addition
\begin{equation}\label{relabel}
	\delta^{\alpha} \|\omega_0^{\delta}\|_{L^2}^2  \rightarrow 0,\quad \text{ as } \delta\rightarrow 0.
\end{equation}
This can always be achieved by relabeling the parameter. Note that if $\omega_0\in L^2$, then $(\omega_0^{\delta})_{\delta>0}$ is uniformly bounded in $L^2$ and \eqref{relabel} holds for free.

Our task now is to construct approximate solutions for \eqref{Boussinesq}. For any $\delta>0$, consider the following viscous Boussinesq system with smooth stochastic perturbation and initial data:
\begin{equation}\label{regular-model}
	\left\{\aligned
	& \partial_t \omega^{\delta} + u^{\delta}\cdot \nabla\omega^{\delta} + \sum_{k}\sigma_k^\delta \cdot \nabla\omega^{\delta} \,\dot{W}^k =\partial_1\theta^{\delta} + \Big(\delta+\frac{c_\delta}{2}\Big)\Delta \omega^{\delta},\\
	& \partial_t\theta^{\delta} + u^{\delta}\cdot \nabla\theta^{\delta} = \Delta \theta^{\delta}, \\
	&\omega^{\delta}(0,\cdot) = \omega_0^{\delta}, ~\theta^{\delta}(0,\cdot) = \theta_0^{\delta},
	\endaligned \right.
\end{equation}
where $u^{\delta}=K*\omega^{\delta}$ and $c_\delta$ is the constant in \eqref{c-delta}. Since this model is a stochastic 2D Navier-Stokes equation coupled with a transport-diffusion equation, the well-posedness is as expected. We leave the proof of the following result in Appendix A.

\begin{lem}\label{wellpoesd-reg-model}
	For every filtered probability space $(\Omega,\mathcal{F},(\mathcal{F}_t)_t,\P)$ and every sequence $(W^k)_k$ of independent real Brownian motions, there exists a unique global smooth solution to \eqref{regular-model}, that is, two $(\mathcal{F}_t)_t$-adapted $\dot{H}^{-1}\cap \dot{H}^{m}$-valued continuous processes $\omega^\delta$ and $\theta^\delta$ such that $\P\text{-a.s.}$, for every $t\in [0,T]$ and $x\in \R^2$,
	\begin{align*}
		\omega^\delta_t &= \omega_0^\delta
		-\int_{0}^{t}\mathrm{div}(u_s^\delta\omega_s^\delta)\, \mathrm{d}s
		-\sum_{k}\int_{0}^{t}\mathrm{div}(\sigma_k^\delta\omega_s^\delta) \, \d W_s^k
		+ \int_{0}^{t} \partial_1\theta_s^\delta  \, \d s
		+\Big(\delta+\frac{c_\delta}{2}\Big)\int_{0}^{t} \Delta \omega_s^{\delta} \, \d s,\\
		\theta_t^{\delta}&= \theta_0^{\delta}
		-\int_{0}^{t}\mathrm{div}(u_s^\delta\theta_s^\delta)\, \mathrm{d}s
		+  \int_{0}^{t} \Delta \theta_s^{\delta} \, \d s,
	\end{align*}
    where $u^{\delta}=K*\omega^{\delta}$ and $m\in \N, T\in \R_+$ are arbitrarily given finite numbers. We also have the following energy estimates: $\P$-a.s. $\forall t\in [0,T]$,
    \begin{equation*}
    	\begin{split}
    		&\|\theta_t^{\delta}\|_{L^2}^2 + 2\int_{0}^{t} \|\nabla\theta_s^{\delta}\|_{L^2}^2 \,\d s =\|\theta_0^{\delta}\|_{L^2}^2,\quad \|\omega_t^{\delta}\|_{L^2}\leq \|\omega_0^{\delta}\|_{L^2} +\sqrt{t}\|\theta_0^{\delta}\|_{L^2},
    	\end{split}
    \end{equation*}
and moreover,
    \begin{equation*}
    	\begin{split}
    		&\E\Big[ \|\omega^{\delta}\|_{C_T(\dot{H}^{-1}\cap \dot{H}^{m} )}^2 \Big]
    		+ \E\Big[ \|\theta^{\delta}\|_{C_T(\dot{H}^{-1}\cap \dot{H}^{m} )}^2 \Big]
    		\leq C\big(T,m,\delta, \|\omega_0^{\delta}\|_{\dot{H}^{-1}\cap \dot{H}^{m}},\|\theta_0^{\delta}\|_{\dot{H}^{-1}\cap \dot{H}^{m}}\big).
    	\end{split}
    \end{equation*}
\end{lem}

\subsection{Uniform energy estimate}
Now  we turn to deriving uniform energy estimates for the approximate solutions $(\omega^\delta,\theta^\delta)_\delta$. We start with the estimate of $\dot{H}^{-1}$ norm of vorticity, where an additional $H^{-\alpha}$ regularity due to the (rough) Kraichnan noise will emerge.
Let $G$ be the Green function of $-\Delta$ on $\R^2$, namely $-\Delta G=\delta_0$, then $u^{\delta}=K*\omega^{\delta}=\nabla^{\perp}G*\omega^{\delta}$, hence we have
\begin{equation*}
	\|\omega^{\delta}\|_{\dot{H}^{-1}}^2 = \int_{\mathbb{R}^2}|\xi|^{-2}\, \big|\widehat{\omega^{\delta}}(\xi)\big|^2 \,\d \xi = 4\pi^2\langle \omega^{\delta},G*\omega^{\delta} \rangle= 4\pi^2 \|u^{\delta}\|_{L^2}^2.
\end{equation*}

\begin{prop}\label{key-bdd}
	For every $T>0$, as $\delta\to 0$,
	\begin{equation*}
		\begin{split}
			\E\bigg[\sup_{t\in [0,T]} \|\omega_t^{\delta}\|_{\dot{H}^{-1}}^2 \bigg]
			+\int_{0}^{T}\E\big[\|\omega_t^{\delta}\|_{\dot{H}^{-\alpha}}^2 \big]\, \d t
			&\lesssim_{\alpha, T} \|\omega_0\|_{\dot{H}^{-1}}^2 +  \|\theta_0\|_{L^2}^2 + o(1),
		\end{split}
	\end{equation*}
where the constant behind the inequality depends only on $\alpha$ and $T$, independent of $\delta>0$.
\end{prop}

\begin{rmk}
In the recent work \cite{GGM}, the underlying constant depending on the parameter $\alpha$ was explicitly calculated by means of Mellin transform.
\end{rmk}

\begin{proof}
Since $\omega^{\delta} \in \dot H^{-1} \cap \dot H^m$ is smooth, we directly apply It\^o's formula on $\dot{H}^{-1}$ (cf. \cite[Theorem 4.32]{Daprato}) and get
	\begin{equation}\label{Ito-reg-w-H^-1}
		\begin{split}
			\d \big\langle \omega^{\delta}, G*\omega^{\delta} \big\rangle
			& = -2\big\langle G*\omega^{\delta},u^{\delta}\cdot \nabla \omega^{\delta} \big\rangle \,\d t + 2\big\langle G*\omega^{\delta}, \partial_1\theta^{\delta}\big\rangle \,\d t
			+ 2\delta\big\langle G*\omega^{\delta},\Delta \omega^{\delta} \big\rangle \,\d t \\
			&\quad- 2\sum_{k}\big\langle G*\omega^{\delta}, \sigma_k^{\delta}\cdot \nabla \omega^{\delta}  \big\rangle \,\d W^k  \\
			&\quad+  c_{\delta}\big\langle G*\omega^{\delta}, \Delta\omega^{\delta}\big\rangle \,\d t
			+ \sum_{k}\big\langle G*(\sigma_k^{\delta}\cdot \nabla\omega^{\delta}), \sigma_k^{\delta}\cdot \nabla\omega^{\delta}\big\rangle \,\d t.
		\end{split}
	\end{equation}

	The nonlinear term $\big\langle G*\omega^{\delta},u^{\delta}\cdot \nabla \omega^{\delta} \big\rangle = -\big\langle (\nabla G*\omega^{\delta})\cdot u^{\delta},  \omega^{\delta} \big\rangle=0$ as expected. It is obvious that the viscous term $\delta\big\langle G*\omega^{\delta},\Delta \omega^{\delta} \big\rangle = -\delta \|\omega^{\delta}\|_{L^2}^2\leq 0$ and
	\begin{equation*}
		\big| \big\langle G*\omega^{\delta}, \partial_1\theta^{\delta}\big\rangle \big|\lesssim \|\omega^{\delta}\|_{\dot{H}^{-1}} \|\theta^{\delta}\|_{L^2}\lesssim \|\omega^{\delta}\|_{\dot{H}^{-1}}^2+\|\theta^{\delta}\|_{L^2}^2.
	\end{equation*}

    To take expectation in \eqref{Ito-reg-w-H^-1}, one needs to check the  integrability of the martingale part; indeed, by integration by parts, Lemma \ref{reg-Q} with $m=n=0$ and Young's inequality,
	\begin{align*}
	\sum_{k}\E\Big[\big|\big\langle G*\omega^{\delta}, \sigma_k^{\delta}\cdot \nabla\omega^{\delta} \big\rangle\big|^2 \Big]
	&= \sum_{k}\E\Big[\big|\big\langle \sigma_k^{\delta}\cdot  (\nabla G*\omega^{\delta}), \omega^{\delta} \big\rangle\big|^2 \Big]   \\
	&= \E \sum_{k}\iint  \sigma_k^{\delta}(x)\cdot  \nabla G*\omega^{\delta}(x) \omega^{\delta}(x)\,
	\sigma_k^{\delta}(y)\cdot  \nabla G*\omega^{\delta}(y) \omega^{\delta}(y)\,\d x\d y \\
	&=  \E\Big[\big\langle Q^{\delta}*((\nabla G*\omega^{\delta})\,\omega^{\delta}),  (\nabla G*\omega^{\delta})\,\omega^{\delta} \big\rangle \Big]  \\
	&\leq \|Q^{\delta}\|_{L^{\infty}} \E\Big[ \|(\nabla G*\omega^{\delta})\,\omega^{\delta}\|_{L^1}^2 \Big]   \\
	&\leq \|\widehat{Q}\|_{L^1} \E\Big[ \| \omega^{\delta}\|_{L^2}^2 \|\omega^{\delta}\|_{\dot{H}^{-1}}^2 \Big]< \infty,
	\end{align*}
    where the last step follows from Lemma \ref{wellpoesd-reg-model}, hence the stochastic part vanishes once taking expectation.

     To handle with the last two terms in \eqref{Ito-reg-w-H^-1} which are from noise,  recall $Q^{\delta}(0)=c_{\delta}I_2$, we have
     \begin{equation*}
     	 c_{\delta}\big\langle G*\omega^{\delta}, \Delta\omega^{\delta}\big\rangle
     	 = \big\langle \tr\big[Q^{\delta}(0)D^2G\big]*\omega^{\delta},\omega^{\delta} \big\rangle
     \end{equation*}
     and
     \begin{equation}\label{sim-conpute}
     	\begin{split}
     		\sum_k\big\langle G*(\sigma_k^{\delta}\cdot\nabla\omega^{\delta}), \sigma_k^{\delta}\cdot\nabla\omega^{\delta} \big\rangle
     		&=\sum_{k}\big\langle \nabla\cdot(\sigma_k^{\delta}\omega^{\delta}),\nabla\cdot \big(G*(\sigma_k^{\delta}\omega^{\delta})\big)\big\rangle \\
     		&=-\sum_{k}\big\langle \sigma_k^{\delta}\omega^{\delta},D^2 G*(\sigma_k^{\delta}\omega^{\delta})\big\rangle \\
     		&=-\sum_{i,j}\iint\partial_{ij}^2G(x-y)\sum_{k}\sigma_k^{\delta,i}(x)\sigma_k^{\delta,j}(y) \omega^{\delta}(x)\omega^{\delta}(y)\,\d x\d y \\
     		&=-\big\langle\mathrm{tr}\big[Q^{\delta}D^2G\big]*\omega^{\delta},\omega^{\delta}\big\rangle.
     	\end{split}
     \end{equation}
     So we arrive at
     \begin{equation}\label{sim2}
     	\begin{split}
     		&c_{\delta}\big\langle G*\omega^{\delta}, \Delta\omega^{\delta}\big\rangle
     		+ \sum_{k}\big\langle G*(\sigma_k^{\delta}\cdot \nabla\omega^{\delta}), \sigma_k^{\delta}\cdot \nabla\omega^{\delta}\big\rangle \\
     		&= \big\langle\tr\big[(Q^{\delta}(0)-Q^{\delta})D^2G\big]*\omega^{\delta},\omega^{\delta}\big\rangle \\
     		&:= \big\langle J*\omega^{\delta},\omega^{\delta} \big\rangle.
     	\end{split}
     \end{equation}

     Now we proceed as in \cite[Section 4]{CogMau}. Introduce a bump function $\varphi$ (see Section 1.4), and decompose $J$ into the following terms:
     \begin{align*}
     	J&= \tr\big[(Q(0)-Q)D^2G\big]\varphi  \\
     	&\quad+ \tr\big[ \big((Q^{\delta}(0)-Q^{\delta})-(Q(0)-Q)\big)D^2G\big]\varphi\\
     	&\quad+\tr\big[(Q^{\delta}(0)-Q^{\delta})D^2G\big](1-\varphi)\\
     	&:= A + R_1 + R_3;
     \end{align*}
we adopt the same notation as in \cite{CogMau} to facilitate the citation of the results there. From \cite[Lemmas 4.3--4.5]{CogMau} (with $\epsilon=\alpha$ there), we have the estimates
      \begin{align*}
      	&\widehat{A}(\xi)\leq -c_{\alpha}\langle \xi\rangle^{-2\alpha}+C\langle \xi\rangle^{-2},\quad
      	\widehat{R_3}(\xi)\leq C\vert \xi\vert^{-2}, \quad \forall\xi\in \R^2,\\
      	&|R_1(x)|\leq C\delta^{\alpha}|x|^{-2+\alpha}\varphi(x),\quad \forall x\in \R^2,
      \end{align*}
where $c_{\alpha},C$ are two constants depending only on $\alpha\in (0,1/2)$. Thus by Young's inequality, we obtain that
      \begin{equation*}
      	\begin{split}
      		\big\langle J*\omega^{\delta},\omega^{\delta} \big\rangle
      		&= \big\langle A*\omega^{\delta},\omega^{\delta} \big\rangle
      		+ \big\langle R_3*\omega^{\delta},\omega^{\delta} \big\rangle
      		+\big\langle R_1*\omega^{\delta},\omega^{\delta} \big\rangle \\
      		&= \int \big( \widehat{A}(\xi) + \widehat{R_3}(\xi)\big) \big|\widehat{\omega^{\delta}}(\xi)\big|^2 \,\d \xi
      		+\iint R_1(x-y) \omega^{\delta}(x) \omega^{\delta}(y)\,\d x\d y\\
      		&\leq -c_{\alpha}\|\omega^{\delta}\|_{H^{-\alpha}}^2 + C\|\omega^{\delta}\|_{\dot{H}^{-1}}
      		+ C\delta^{\alpha}\|\omega^{\delta}\|_{L^2}^2 \\
      		&\leq -c_{\alpha}\|\omega^{\delta}\|_{\dot{H}^{-\alpha}}^2 + C\|\omega^{\delta}\|_{\dot{H}^{-1}}^2
      		+ o(1),
      	\end{split}
      \end{equation*}
where the last line follows from $\|\omega^{\delta}\|_{\dot{H}^{-\alpha}}^2\lesssim \|\omega^{\delta}\|_{\dot{H}^{-1}}^2 + \|\omega^{\delta}\|_{H^{-\alpha}}^2$, Lemma \ref{wellpoesd-reg-model} and the assumption \eqref{relabel}.

      Substituting the above estimates into \eqref{Ito-reg-w-H^-1}, integrating in time and taking expectation, we get, for every $t\in [0,T]$,
      \begin{align*}
      	\mathbb{E}\big[\Vert\omega_t^{\delta}\Vert_{\dot{H}^{-1}}^2\big] +\int_{0}^{t}\mathbb{E}\left[\Vert\omega_s^{\delta}\Vert_{\dot{H}^{-\alpha}}^2\right]\d s
      	\lesssim_{\alpha} \Vert\omega_0^{\delta}\Vert_{\dot{H}^{-1}}^2
      	+ \int_{0}^{t}\mathbb{E}\big[\Vert\omega_s^{\delta}\Vert_{\dot{H}^{-1}}^2\big]\,\d s
      	+ \int_{0}^{t} \E\big[ \|\theta_s^{\delta}\|_{L^2}^2 \big]\,\d s+o(1).
      \end{align*}
      Then by the $L^2$-bounds in Lemma \ref{wellpoesd-reg-model}, convergence of initial data \eqref{converge-initial-1}, \eqref{converge-initial-2} and Gr\"onwall's inequality, we obtain
      \begin{equation}\label{H^-1-mid}
      	\begin{split}
      		\sup_{t\in [0,T]}\E\Big[ \|\omega_t^{\delta}\|_{\dot{H}^{-1}}^2 \Big]
      		+\int_{0}^{T}\E\big[\|\omega_t^{\delta}\|_{\dot{H}^{-\alpha}}^2 \big]\, \d t
      		&\lesssim_{\alpha, T} \|\omega_0\|_{\dot{H}^{-1}}^2 +  \|\theta_0\|_{L^2}^2 + o(1).
      	\end{split}
    \end{equation}

To complete the proof, we need to estimate the $C\big([0,T],\dot{H}^{-1}\big)$ norm of $\omega^{\delta}$; by \eqref{Ito-reg-w-H^-1}, the computations are similar as above, except the martingale part, thus
    \begin{align*}
    	\E\bigg[\sup_{t\in [0,T]} \|\omega^{\delta}_t\|_{\dot{H}^{-1}}^2 \bigg]
    	&\leq \|\omega^{\delta}_0\|_{\dot{H}^{-1}}^2 + C\int_{0}^{T}\Big(\E\big[\|\omega^{\delta}_t\|_{\dot{H}^{-1}}^2 \big] +  \E\big[\|\theta_t^{\delta}\|_{L^2}^2\big]\Big)\, \d t \\
    	&\quad +\E\bigg[\sup_{t\in [0,T]}\bigg|\sum_{k}\int_{0}^{t}\big\langle G*\omega_s^{\delta}, \sigma_k^{\delta}\cdot\nabla\omega_s^{\delta}  \big\rangle \,\d W_s^k\bigg|\bigg].
    \end{align*}
    To bound the stochastic integral term, from the Burkholder-Davis-Gundy inequality, we have
    \begin{align*}
    	M:=\E\bigg[\sup_{t\in [0,T]}\bigg|\sum_{k}\int_{0}^{t}\big\langle G*\omega_s^{\delta}, \sigma_k^{\delta}\cdot\nabla\omega_s^{\delta}  \big\rangle \,\d W_s^k\bigg|\bigg]
    	&\lesssim \E\Bigg[\bigg(\int_{0}^{T}\sum_{k}\big|\big\langle G*\omega_t^{\delta}, \sigma_k^{\delta}\cdot\nabla\omega_t^{\delta} \big\rangle\big|^2 \,\d t\bigg)^\frac{1}{2}\Bigg].
    \end{align*}
By elementary calculations, one can show the following identity:
    \begin{equation}\label{identity}
    	\sigma_k^{\delta}\cdot \nabla \omega^{\delta}= \curl \big((\sigma_k^{\delta}\cdot \nabla)u^{\delta} + (D\sigma_k^{\delta})^T u^{\delta}\big),\quad\text{where } (D\sigma_k^{\delta})_{ij}^T=\partial_i\sigma_k^{\delta,j}.
    \end{equation}
    Hence by integration by parts, we have
    \begin{equation*}
    	\begin{split}
    		\sum_{k}|\langle G*\omega^{\delta}, \sigma_k^{\delta}\cdot\nabla\omega^{\delta} \rangle|^2
    		&=\sum_{k}\Big|\Big\langle G*\omega^{\delta}, \curl \big((\sigma_k^{\delta}\cdot \nabla)u^{\delta} + (D\sigma_k^{\delta})^T u^{\delta}\big) \Big\rangle\Big|^2\\
    		&=\sum_{k}
    		|\langle u^{\delta},\sigma_k^{\delta}\cdot\nabla u^{\delta}+ (D\sigma_k^{\delta})^T u^{\delta} \rangle|^2.
    	\end{split}	
    \end{equation*}
Since $\div \sigma_k^{\delta}=0$, we have $\langle u^{\delta},\sigma_k^{\delta}\cdot\nabla u^{\delta}\rangle=0$, hence by Lemma \ref{reg-Q} with $m=n=1$, we obtain
    \begin{align*}
    	\sum_{k}|\langle G*\omega^{\delta}, \sigma_k^{\delta}\cdot\nabla\omega^{\delta} \rangle|^2
    	&=\sum_{k}|\langle u^{\delta}, (D\sigma_k^{\delta})^T u^{\delta} \rangle|^2 \\
    	&=\iint \sum_{k}u^{\delta,i}(x)\partial_i\sigma_k^{\delta,j}(x)u^{\delta,j}(x)\,
    	u^{\delta,l}(y)\partial_l\sigma_k^{\delta,m}(y)u^{\delta,m}(y) \,\d x\d y \\
    	&= -\big\langle u^{\delta,i}u^{\delta,j},\partial_{il}Q^{\delta,jm}*(u^{\delta,l}u^{\delta,m})\big\rangle \\
    	&\leq \max_{i,j}\int \big|\widehat{u^{\delta,i}u^{\delta,j}}(\xi)\big|^2~ \<\xi\>^2\big|\widehat{Q}(\xi)\big|\,\d \xi \lesssim \|u^{\delta}\otimes u^{\delta}\|_{H^{-\alpha}}^2,
    \end{align*}
where in the last step we have used \eqref{Q^}.
By Sobolev embedding $L^{\frac{2}{1+\alpha}}\hookrightarrow H^{-\alpha}$, $\dot{H}^{1-\alpha}\hookrightarrow L^{\frac{2}{\alpha}}$ and H\"older's inequality  we have
    \begin{equation*}
    	\|u^{\delta}\otimes u^{\delta}\|_{H^{-\alpha}} \lesssim  	\|u^{\delta}\otimes u^{\delta}\|_{L^{\frac{2}{1+\alpha}}}
    	\lesssim \|u^{\delta}\|_{L^2}  \|u^{\delta}\|_{L^{\frac{2}{\alpha}}}
    	\lesssim \|u^{\delta}\|_{L^2}  \|u^{\delta}\|_{\dot{H}^{1-\alpha}}.
    \end{equation*}
Since $\|u^{\delta}\|_{L^2}^2=(2\pi)^{-2}\|\omega^{\delta}\|_{\dot{H}^{-1}}^2$ and $\|u^{\delta}\|_{\dot{H}^{1-\alpha}}^2=(2\pi)^{-2}\|\omega^{\delta}\|_{\dot{H}^{-\alpha}}^2$, summing up the above arguments, we have, for every $0<\epsilon\ll1$,
    \begin{align*}
    	M
    	&\lesssim \E\Bigg[\bigg(\int_{0}^{T}  \|\omega_t^{\delta}\|_{\dot{H}^{-1}}^2  \|\omega_t^{\delta}\|_{\dot{H}^{-\alpha}}^2  \,\d t\bigg)^\frac{1}{2}\Bigg]   \\
    	& \leq\E\Bigg[  \sup_{t\in [0,T]} \|\omega_t^{\delta}\|_{\dot{H}^{-1}}\bigg(\int_{0}^{T}    \|\omega_t^{\delta}\|_{\dot{H}^{-\alpha}}^2  \,\d t\bigg)^\frac{1}{2}\Bigg] \\
    	&\leq \epsilon\, \E\bigg[\sup_{t\in [0,T]} \|\omega_t^{\delta}\|_{\dot{H}^{-1}}^2 \bigg]
    	+ C_{\epsilon} \int_{0}^{T}\E\big[\|\omega_t^{\delta}\|_{\dot{H}^{-\alpha}}^2 \big]
    	\, \d t,
    \end{align*}
    which complete the proof combined with \eqref{H^-1-mid}.
\end{proof}

Notice that we have derived the bounds on $\|\omega^\delta\|_{C_T\dot{H}^{-1}},\|\omega^\delta\|_{L_T^2 \dot{H}^{-\alpha}}$ and $\|\theta^{\delta}\|_{C_TL^2\cap L_T^2H^1}$ which are uniform in $\delta$. Now we show how to obtain uniform control of $\|\omega^\delta\|_{L_T^{\infty}L^{p}}.$ 
The $L^p$ bound of vorticity is conservative since the drift term and noise term are divergence free and the dissipation term $\Delta\omega^{\delta}$ does not increase this bound. Indeed, we have for every $1\leq p<\infty$ (valid also for $L^{\infty}$ norm),
\begin{equation}\label{w-Lp-mid}
	 \|\omega_t^{\delta} \|_{L^{p}} \le \|\omega_0^{\delta} \|_{L^p}+\int_{0}^{t}\|\nabla\theta_s^{\delta}\|_{L^{p}} \,\d s,\quad t\in [0,T],
\end{equation}
see \cite[Proposition 3.1]{GalLuo23} for a detailed proof. The above inequality shows that we should bound $L_T^1L^p$ norm of gradient of temperature  to control the vorticity.

\subsubsection*{Case 1: $\omega_0\in L^2$}

If $\omega_0\in L^2$, then from Lemma \ref{wellpoesd-reg-model}, \eqref{converge-initial-1} and \eqref{converge-initial-2}, we have the  uniform $L^2$ bound: as $\delta\rightarrow 0$,
\begin{equation}\label{case1}
	\|\omega_t^{\delta}\|_{L^2}\leq \|\omega_0\|_{L^2} +\sqrt{t}\|\theta_0\|_{L^2} + o(1),\quad \forall t\in [0,T], ~\P\text{-a.s.}
\end{equation}

\subsubsection*{Case 2: $\omega_0\in L^p,~1\leq p<2$}

We can rewrite the equation for $\theta^{\delta}$ in mild form in the current smooth setting, that is,
\begin{equation*}
	\theta_t^{\delta}=e^{t\Delta}\theta_0^{\delta} -\int_{0}^{t} e^{(t-s)\Delta} \big(u_s^{\delta}\cdot\nabla\theta_s^{\delta}\big)\,\d s,\quad \forall t\in [0,T],
\end{equation*}
where $(e^{t\Delta})_{t\geq 0}$ is the heat semigroup generated by $\Delta$. With Young's inequality, we have for all $t\in [0,T]$,
\begin{align*}
    \|\nabla\theta_t^{\delta}\|_{L^1}
    &\leq \|\nabla e^{t\Delta}\theta_0^{\delta}\|_{L^1}
    + \int_{0}^{t} \big\|\nabla e^{(t-s)\Delta} \big(u_s^{\delta}\cdot\nabla\theta_s^{\delta}\big)\big\|_{L^1}\,\d s  \\
    &\lesssim  t^{-\frac{1}{2}} \|\theta_0^{\delta}\|_{L^1}
    + \int_{0}^{t} (t-s)^{-\frac{1}{2}}  \|u_s^{\delta}\|_{L^2}\,\|\nabla\theta_s^{\delta}\|_{L^2}\,\d s \\
    &\lesssim t^{-\frac{1}{2}} \|\theta_0^{\delta}\|_{L^1}
    + \|u^{\delta}\|_{L_T^{\infty}L^2}\int_{0}^{t} (t-s)^{-\frac{1}{2}}  \,\|\nabla\theta_s^{\delta}\|_{L^2}\,\d s.
\end{align*}
Integrating over $[0,T]$ and by Fubini's theorem, we get $\P\text{-a.s.}$,
\begin{equation}\label{tidu-theta-L1}
	\begin{split}
		\int_{0}^{T} \|\nabla\theta_t^{\delta}\|_{L^1}\,\d t
		&\lesssim   T^{\frac{1}{2}} \|\theta_0^{\delta}\|_{L^1} + \|\omega^{\delta}\|_{C_T\dot{H}^{-1}}\int_{0}^{T} (T-s)^{\frac{1}{2}}  \,\|\nabla\theta_s^{\delta}\|_{L^2}\,\d s \\
		&\lesssim T^{\frac{1}{2}} \|\theta_0^{\delta}\|_{L^1} + T \|\omega^{\delta}\|_{C_T\dot{H}^{-1}}\bigg(\int_{0}^{T}   \|\nabla\theta_s^{\delta}\|_{L^2}^2\,\d s\bigg)^{1/2} \\
		&\lesssim_T  \|\theta_0^{\delta}\|_{L^1} +  \|\theta_0^{\delta}\|_{L^2} \|\omega^{\delta}\|_{C_T\dot{H}^{-1}}.
	\end{split}
\end{equation}
Combining Lemma \ref{wellpoesd-reg-model}, \eqref{w-Lp-mid} and  \eqref{tidu-theta-L1}, we arrive at, $\P\text{-a.s.}$,
\begin{equation}\label{case2}
	\begin{split}
		\|\omega^{\delta}\|_{L_T^{\infty}L^{p}}
		&\leq  \|\omega_0^{\delta}   \|_{L^p}+\int_{0}^{T}\|\nabla\theta_s^{\delta}\|_{L^{p}} \,\d s \\
		&\leq \|\omega_0^{\delta} \|_{L^p}+\int_{0}^{T}\|\nabla\theta_s^{\delta}\|_{L^{1}\cap L^2} \,\d s  \\
		&\leq \|\omega_0^{\delta} \|_{L^p}
		+ T^{\frac{1}{2}} \|\theta_0^{\delta}\|_{L^2} + T^{\frac{1}{2}} \|\theta_0^{\delta}\|_{L^1}
		+T \|\theta_0^{\delta}\|_{L^2} \|\omega^{\delta}\|_{C_T\dot{H}^{-1}}          \\
		&= C\Big(T, \|\omega_0 \|_{L^p},~ \|\theta_0\|_{L^1\cap L^2},~\|\omega^{\delta}\|_{C_T\dot{H}^{-1}}\Big),
	\end{split}
\end{equation}
where in this section we use the notation $C\big(T, \|\omega_0 \|_{L^p},~ \|\theta_0\|_{L^1\cap L^2},~\|\omega^{\delta}\|_{C_T\dot{H}^{-1}}\big)$ to denote a $\R_+$-valued random variable depending on $\|\omega^{\delta}\|_{C_T\dot{H}^{-1}}$ linearly and initial data.
\subsubsection*{Case 3: $\omega_0\in L^1\cap L^p,~ 2<p<\infty$}
Since $\omega_0\in L^1\cap L^p\subset L^1\cap L^2$ for every $p>2$, we have
\begin{equation*}
	\|\omega^{\delta}\|_{L_T^{\infty}(L^1\cap L^2)}\leq C\Big(T, \|\omega_0 \|_{L^1\cap L^2}, \|\theta_0\|_{L^1\cap L^2},\|\omega^{\delta}\|_{C_T\dot{H}^{-1}}\Big).
\end{equation*}

Now choosing $1<l<2<q<\infty,\frac{1}{q}=\frac{1}{l}-\frac{1}{2}$ such that $p<q$, we have $\|u^\delta\|_{L^q}\lesssim \|\omega^\delta\|_{L^l}$ by Sobolev embedding. Using heat kernel estimates and H\"older's inequality, we get for all $t\in [0,T]$,
\begin{align*}
	 \|\nabla\theta_t^{\delta}\|_{L^p}
	&\leq \|\nabla e^{t\Delta}\theta_0^{\delta}\|_{L^p}
	+ \int_{0}^{t} \big\|\nabla e^{(t-s)\Delta} \big(u_s^{\delta}\cdot\nabla\theta_s^{\delta}\big)\big\|_{L^p}\,\d s  \\
&\lesssim t^{-1+\frac{1}{p}} \|\theta_0^{\delta}\|_{L^2} + \int_0^t (t-s)^{-\frac12- \frac1l + \frac1p} \|u_s^{\delta}\cdot\nabla\theta_s^{\delta}\|_{L^l}\, ds \\
	&\lesssim  t^{-1+\frac{1}{p}} \|\theta_0^{\delta}\|_{L^2}
	+ \int_{0}^{t} (t-s)^{-1+\frac{1}{p}-\frac{1}{q}}  \|u_s^{\delta}\|_{L^q}\,\|\nabla\theta_s^{\delta}\|_{L^2}\,\d s ; \end{align*}
since $\|u_s^{\delta}\|_{L^q}\lesssim \|\omega^\delta_s\|_{L^l} \le \|\omega^\delta\|_{L^\infty_T L^l} \lesssim \|\omega^\delta\|_{L^\infty_T (L^1\cap L^2)}$, we have
\begin{align*}
	\|\nabla\theta_t^{\delta}\|_{L^p}
	& \lesssim t^{-1+\frac{1}{p}} \|\theta_0^{\delta}\|_{L^2} + \|\omega^{\delta}\|_{L_T^{\infty}(L^1\cap L^2)}\int_{0}^{t} (t-s)^{-1+\frac{1}{p}-\frac{1}{q}}  \,\|\nabla\theta_s^{\delta}\|_{L^2}\,\d s.
\end{align*}
Integrating over $[0,T]$ and by Fubini's theorem, we obtain
\begin{align*}
	\int_{0}^{T} \|\nabla\theta_t^{\delta}\|_{L^p}\,\d t
	&\lesssim_p   T^{\frac{1}{p}} \|\theta_0^{\delta}\|_{L^2} + \|\omega^{\delta}\|_{L_T^{\infty}(L^1\cap L^2)}\int_{0}^{T} (T-s)^{\frac{1}{p}-\frac{1}{q}}  \,\|\nabla\theta_s^{\delta}\|_{L^2}\,\d s \\
	&\leq T^{\frac{1}{p}} \|\theta_0^{\delta}\|_{L^2} + \|\omega^{\delta}\|_{L_T^{\infty}(L^1\cap L^2)} T^{\frac{1}{p}-\frac{1}{q}+\frac{1}{2}}
	\|\nabla\theta^{\delta}\|_{L_T^2L^2}          \\
	&\leq C\Big(p, T,\|\theta_0^{\delta}\|_{L^2},~  \|\omega^{\delta}\|_{L_T^{\infty}(L^1\cap L^2)}\Big)\\
	&\leq  C\Big(p, T,\|\omega_0\|_{L^1\cap L^2},~\|\theta_0\|_{L^1\cap L^2},~  \|\omega^{\delta}\|_{C_T\dot{H}^{-1}}\Big).
\end{align*}
As a result, we obtain the  bound
\begin{equation}\label{case3}
	\|\omega^{\delta}\|_{L_T^{\infty}(L^1\cap L^p)}\leq C\Big(T, \|\omega_0 \|_{L^1\cap L^p}, \|\theta_0\|_{L^1\cap L^2},\|\omega^{\delta}\|_{C_T\dot{H}^{-1}}\Big), \quad \P\text{-a.s.}
\end{equation}

The estimates \eqref{case1}, \eqref{case2} and \eqref{case3} are the desired bounds.
We end this section by deriving the bound for $\|\theta^{\delta}\|_{\dot{H}^{-1}}$ uniform in $\delta$.
Due to the high regularity of $\theta^{\delta}$, we can compute
\begin{align}\label{theta-delta-H^-1}
	\frac{\d}{\d t} \< \theta^{\delta},G*\theta^{\delta} \> + 2\|\theta^{\delta}\|_{L^2}^2 = -2\big\< u^{\delta}\cdot \nabla\theta^{\delta}, G*\theta^{\delta}\big\>.
\end{align}
By integration by parts and Lemma \ref{product} with $a=2\alpha, b=1-\alpha$, we get
\begin{align*}
	\big|\big\< u^{\delta}\cdot \nabla\theta^{\delta}, G*\theta^{\delta}\big\>\big|
	&=
	\big|\big\< u^{\delta}\cdot \big(\nabla G*\theta^{\delta}\big), \theta^{\delta} \big\>\big| \\
	&\lesssim
	\|u^{\delta}\|_{\dot{H}^{2\alpha}}
	\|\nabla G*\theta^{\delta}\|_{\dot{H}^{1-\alpha}}
	\|\theta^{\delta}\|_{\dot{H}^{-\alpha}} \\
	&\leq
	\|\omega^{\delta}\|_{\dot{H}^{2\alpha-1}}
	\|\theta^{\delta}\|_{\dot{H}^{-\alpha}}^2   \\
	&\leq
	\|\theta^{\delta}\|_{L^2}^2
	+ C_{\alpha}\|\omega^{\delta}\|_{\dot{H}^{2\alpha-1}}^{1/\alpha}
	\|\theta^{\delta}\|_{\dot{H}^{-1}}^2,
\end{align*}
where in the last step we have used interpolation $\dot{H}^{-1}\cap L^2\subset \dot{H}^{-\alpha}$ and  Young's inequality.

Recall $0<\alpha<1-\frac{1}{p\wedge2}$, then from Sobolev embedding $L^{p\wedge 2}\hookrightarrow \dot{H}^{1-\frac{2}{p\wedge 2}}$, and interpolation $\dot{H}^{-1}\cap \dot{H}^{1-\frac{2}{p\wedge 2}}\subset \dot{H}^{2\alpha-1}$, we have $\P\text{-a.s. }$,
\begin{align*}
	\|\omega_t^{\delta}\|_{\dot{H}^{2\alpha-1}}
	\leq \|\omega_t^{\delta}\|_{\dot{H}^{-1}}
	+ \|\omega_t^{\delta}\|_{L^{p\wedge 2}}
	\leq  C\Big(\|\omega^{\delta}\|_{C([0,t]; \dot{H}^{-1})}\Big), \quad \forall t\in [0,T],
\end{align*}
where in the last step we have used  the pathwise bounds \eqref{case1} and \eqref{case2}, omitting the dependence on time and initial data. Combining these estimates with \eqref{theta-delta-H^-1}, we obtain
\begin{equation*}
	\frac{\d}{\d t}  \|\theta^{\delta}\|_{\dot{H}^{-1}}^2 + \|\theta^{\delta}\|_{L^2}^2
	\leq C^{\frac{1}{\alpha}}\Big(\|\omega^{\delta}\|_{C([0,t]; \dot{H}^{-1})}\Big) \|\theta^{\delta}\|_{\dot{H}^{-1}}^2.
\end{equation*}
By Gr\"onwall's inequality, we arrive at
\begin{equation}\label{theta-delta-H^-1-bdd1}
	\|\theta^{\delta}\|_{C_T\dot{H}^{-1}}^2
	\leq e^{TC^{\frac{1}{\alpha}}\big(\|\omega^{\delta}\|_{C_T \dot{H}^{-1}}\big)} \|\theta_0^{\delta}\|_{\dot{H}^{-1}}^2, \quad \P\text{-a.s.}
\end{equation}

\section{Almost sure convergence in $C_T \dot{H}^{-1}$} \label{Converge}

Instead of showing tightness of the laws of $(\omega^{\delta},\theta^{\delta})_{\delta}$ and exploiting the Skorohod representation theorem to get almost sure convergence in $C_T\dot{H}^{-1}$ on a new probability space, in this section we  directly show that there exists a subsequence  $(\omega^{\delta_n},\theta^{\delta_n})_{n\in \N}$ which is Cauchy in space $C_T\dot{H}^{-1}\times C_T\dot{H}^{-1}$ for almost every sample path.

Due to some technical reasons, in this section we temporarily assume that $\omega_0\in L^1\cap L^p,\,p\geq 2$, in which case $\|\omega_0^{\delta}\|_{L^2}$ is uniformly bounded. This will facilitate the estimate of \eqref{tech-bdd}.

\begin{prop}\label{converge-H^-1}
	Assume $\omega_0\in \dot{H}^{-1}\cap L^1\cap L^p$ with $p\geq 2$, then for every $\epsilon>0$, we have
	\begin{equation*}
		\lim_{\delta_1,\delta_2\rightarrow 0}
	    \P\Big( \|\omega^{\delta_1}-\omega^{\delta_2}\|_{C_{T}\dot{H}^{-1}}>\epsilon \Big)=0,\quad
	    \lim_{\delta_1,\delta_2\rightarrow 0}
	    \P\Big( \|\theta^{\delta_1}-\theta^{\delta_2}\|_{C_{T}\dot{H}^{-1}}>\epsilon \Big)=0.
	\end{equation*}
\end{prop}

\begin{proof}
Without  loss of generality, we assume $\delta_2<\delta_1$, let $\delta=\delta_1-\delta_2$. The strategy is to estimate the evolution of the sum $\|\omega^{\delta_1}-\omega^{\delta_2}\|_{\dot{H}^{-1}} + \|\theta^{\delta_1}- \theta^{\delta_2} \|_{\dot{H}^{-1}}$ by initial data. To simplify notations, we neglect the superscript $\delta$ and denote $(\omega^{\delta_i},\theta^{\delta_i},\sigma_k^{\delta_i},c_{\delta_i})$ as $(\omega^{i},\theta^{i},\sigma_k^{i},c_i), i=1,2$ and $(\omega^{\delta_1}-\omega^{\delta_2}, \theta^{\delta_1}-\theta^{\delta_2}, \sigma_k^{\delta_1}- \sigma_k^{\delta_2}, c_{\delta_1}-c_{\delta_2})$ as $(\omega,\theta,\bar{\sigma}_k,c)$. Accordingly, we have $u^i=K*\omega^i,i=1,2$ and $u=K*\omega$. From \eqref{regular-model}, the difference $(\omega,\theta)$ satisfies
	\begin{equation*}
		\left\{\aligned
		& \partial_t \omega + u^{1}\cdot \nabla\omega + u\cdot \nabla\omega^{2}
		+ \sum_{k}\big(\sigma_k^1 \cdot \nabla\omega + \bar{\sigma}_k\cdot \nabla\omega^{2}\big) \,\dot{W}^k
		=\partial_1\theta + \Big(\delta_1+\frac{c_1}{2}\Big)\Delta \omega +
		\Big(\delta+\frac{c}{2}\Big)\Delta \omega^2, \\
		& \partial_t\theta + u^{1}\cdot \nabla\theta + u\cdot \nabla\theta^{2} = \Delta \theta.
		\endaligned \right.
	\end{equation*}

\textbf{Step 1: \boldmath{$\d \|\omega^1-\omega^2 \|_{\dot{H}^{-1}}^2$}.}   Since $\omega= \omega^1-\omega^2$ is smooth, we apply It\^o's formula on $\dot{H}^{-1}$ and get
\begin{equation*}
	\d\,\< \omega, G*\omega \>
	= \sum_{k} \big(\d M_k^1 + \d M_k^2 \big)
	+ J_1\, \d t +J_2\, \d t+ J_3\, \d t,
\end{equation*}
where
\begin{align*}
    \d M_k^1&=-2\big\< G*\omega, \sigma_k^1 \cdot \nabla\omega\big\>\,\d W^k,\quad
	\d M_k^2=-2\big\< G*\omega, \bar{\sigma}_k \cdot \nabla\omega^2\big\>\,\d W^k,\\
	J_1&=-2\big\< G*\omega,u^1\cdot \nabla \omega \big\>
	-2\big\< G*\omega,u\cdot \nabla \omega^2 \big\>
	+ 2\big\< G*\omega, \partial_1\theta \big\> ,\\
	J_2&=(c_1+2\delta_1)\big\< G*\omega,\Delta\omega \big\>
	+ (c+2\delta)\big\< G*\omega,\Delta\omega^2 \big\>,\\
	J_3&=\sum_k\big\< \sigma_k^1 \cdot \nabla\omega + \bar{\sigma}_k\cdot \nabla\omega^{2}, G*\big(\sigma_k^1 \cdot \nabla\omega + \bar{\sigma}_k\cdot \nabla\omega^{2}\big) \big\>.
\end{align*}

We shall first compute terms $J_2,J_3$ to get the  anomalous regularity from noise. Recalling that $Q^{\delta_1}(0)=c_1I_2$ and $Q^{\delta_1}(0)-Q^{\delta_2}(0)=cI_2$, we have
\begin{align*}
	J_2 &=c_1\big\< G*\omega,\Delta\omega \big\>
	-c\big\< G*\omega^2,\Delta\omega^2 \big\>
	+c\big\< G*\omega^1,\Delta\omega^2 \big\>
	+2\delta_1 \big\< G*\omega,\Delta\omega \big\>+ 2\delta\big\< G*\omega,\Delta\omega^2 \big\>\\
	&= \big\< \mathrm{tr}\big[Q^{\delta_1}(0)D^2G\big]*\omega,\omega \big\> + \big\< \mathrm{tr}[\big(Q^{\delta_2}-Q^{\delta_1}\big)(0)D^2G]*\omega^2,\omega^2 \big\>
	 +c\<G*\omega^1,\Delta\omega^2 \>  \\
	&\quad+ 2\delta_1\big\< G*\omega,\Delta\omega \big\>
	+2\delta\big\< G*\omega,\Delta\omega^2\big\>
\end{align*}
and
\begin{align*}
	J_3 &= \sum_k\big\< \sigma_k^1 \cdot \nabla\omega, G*\big(\sigma_k^1 \cdot \nabla\omega \big) \big\>
	+ 2\sum_k\big\< \sigma_k^1 \cdot \nabla\omega,  G*\big(\bar{\sigma}_k\cdot \nabla\omega^{2}\big) \big\>                      \\
	&\quad+ \sum_k\big\< \bar{\sigma}_k\cdot \nabla\omega^{2}, G*\big(\bar{\sigma}_k\cdot \nabla\omega^{2}\big)\big\>.
\end{align*}

Note that  $\sum_k\sigma_k^1\otimes\bar{\sigma}_k
= \sum_k\sigma_k^1\otimes(\sigma_k^{1}-\sigma_k^{2})  =Q*\rho^{\delta_1}*\rho^{\delta_1}-Q*\rho^{\delta_1}*\rho^{\delta_2}=0$, where the last equality follows from the fact that the Fourier transform of $\rho^\delta$ is the indicator function of $B(0,1/\delta)$. A simple computation as \eqref{sim-conpute} yields that
\begin{align*}
	\sum_k\big\< \sigma_k^1 \cdot \nabla\omega,  G*\big(\bar{\sigma}_k\cdot \nabla\omega^{2}\big) \big\> = -\big\<\mathrm{tr}\big[\big(Q*\rho^{\delta_1}*\rho^{\delta_1}-Q*\rho^{\delta_1}*\rho^{\delta_2}\big)D^2G\big]*\omega,\omega^2\big\>=0.
\end{align*}
Then observing that $\sum_{k}\sigma_k^1\otimes\sigma_k^1=Q^{\delta_1}$ and  $~\sum_{k}\bar{\sigma}_k\otimes\bar{\sigma}_k=\sum_{k}\sigma_k^{2}\otimes\sigma_k^{2}-\sum_{k}\sigma_k^{1}\otimes\sigma_k^{1}=Q^{\delta_2}-Q^{\delta_1}$, similarly as \eqref{sim-conpute}, we have
\begin{align*}
	J_3=-\big\<\mathrm{tr}\big[Q^{\delta_1} D^2G\big]*\omega,\omega \big\>
	- \big\< \mathrm{tr}\big[\big(Q^{\delta_2}-Q^{\delta_1}\big)D^2G\big]*\omega^2,\omega^2 \big\>.
\end{align*}
Combining the above expressions of $J_2$ and $J_3$, we obtain
\begin{align*}
	J_2+J_3 &= \big\< \mathrm{tr}\big[\big(Q^{\delta_1}(0)-Q^{\delta_1}\big)D^2G\big]*\omega,\omega \big\>
	+ \big\< \mathrm{tr}\big[\big(\big(Q^{\delta_2}-Q^{\delta_1}\big)(0)-\big(Q^{\delta_2}-Q^{\delta_1}\big)\big)D^2G\big]*\omega^2,\omega^2 \big\> \\
	&\quad+ c\<G*\omega^1,\Delta\omega^2 \>+ 2\delta_1\big\< G*\omega,\Delta\omega \big\> +2\delta\big\< G*\omega,\Delta\omega^2\big\>.
\end{align*}
For the second line of the above expression, note that $\<G*\omega,\Delta\omega \>=-\|\omega\|_{L^2}\leq0$ and by the assumption $\omega_0\in L^2$, \eqref{c-delta} and \eqref{case1}, we have
\begin{align*}
	\big|c\<G*\omega^1,\Delta\omega^2 \>+2\delta\big\< G*\omega,\Delta\omega^2\big\>\big|
	&\lesssim |c_{\delta_1}-c_{\delta_2}|\, \|\omega^1\|_{L^2}\|\omega^2\|_{L^2}
	+ \delta \big(\|\omega^1\|_{L^2}\|\omega^2\|_{L^2}
	+\|\omega^2\|_{L^2}^2\big)\\
	&\leq \delta_1^{2\alpha} \big( \|\omega^1\|_{L^2}^2+\|\omega^2\|_{L^2}^2 \big) = o(1).
\end{align*}
For the first line, as in Proposition \ref{key-bdd}, we have the estimates
\begin{equation*}
	\big\< \mathrm{tr}\big[\big(Q^{\delta_1}(0)-Q^{\delta_1}\big)D^2G\big]*\omega,\omega \big\>
	\leq-c_{\alpha} \|\omega\|_{\dot{H}^{-\alpha}}^2 +C\|\omega\|_{\dot{H}^{-1}}^2 + o(1)
\end{equation*}
and
\begin{equation}\label{err-noise}
	\begin{split}
	& \Big|\tr\big[\big(\big(Q^{\delta_2}-Q^{\delta_1}\big)(0)-\big(Q^{\delta_2}-Q^{\delta_1}\big)\big)D^2G\big] \Big|  \\
	&\leq \sum_{i=1,2} \Big|\tr\big[ \big((Q^{\delta_i}(0)-Q^{\delta_i})-(Q(0)-Q)\big)D^2G\big] \Big| \\
	&\leq \sum_{i=1,2} \Big|\tr\big[ \big((Q^{\delta_i}(0)-Q^{\delta_i})-(Q(0)-Q)D^2G\big)\big]\varphi\Big| \\
	&\quad +\sum_{i=1,2}  \Big|\tr\big[ \big((Q^{\delta_i}(0)-Q^{\delta_i})-(Q(0)-Q)D^2G\big)\big](1-\varphi)\Big| \\
	&=: \sum_{i=1,2}R^i + \sum_{i=1,2}\tilde{R}^i,
	\end{split}
\end{equation}
where $\varphi$ is a bump function. From \cite[Lemma 4.4]{CogMau}, we have the pointwise estimate
  $$\big| \tr\big[ \big((Q^{\delta_i}(0)-Q^{\delta_i})-(Q(0)-Q)\big) D^2G(x)\big]\big|\lesssim \delta_i^{\alpha}|x|^{-2+\alpha},$$
hence by Young's inequality and \eqref{case1}, we have
\begin{equation}\label{tech-bdd}
	\begin{split}
		\big|\big\< R^i*\omega^2, \omega^2\big\> \big|
		&\leq \|R^i\|_{L^1} \|\omega^2\|_{L^2}^2
		\lesssim \delta_i^{\alpha}\|\omega^2\|_{L^2}^2 =o(1);
	\end{split}
\end{equation}
moreover, by \eqref{case2} and Proposition \ref{key-bdd},
\begin{equation*}
	\begin{split}
		\E\Big[\big|\big\< \tilde{R}^i*\omega^2, \omega^2\big\> \big|\Big]
		&\leq \|\tilde{R}^i\|_{L^\infty}\E\Big[ \|\omega^2\|_{L^1}^2\Big]
		\lesssim \delta_i^{\alpha} \E\Big[  \|\omega^2\|_{L^1}^2\Big]  \\
		&\lesssim_T \delta_i^{\alpha}\Big(\|\omega_0\|_{L^1}^2+\|\theta_0\|_{L^1\cap L^2}^2+\|\theta_0\|_{ L^2}^2\E\Big[\|\omega^2\|_{C_T\dot{H}^{-1}}^2\Big] \Big)\\
		&= o(1).
	\end{split}
\end{equation*}

To sum up, we have obtained the estimate
\begin{equation*}
	J_2+J_3\leq -c_{\alpha}\|\omega\|_{\dot{H}^{-\alpha}}^2
		+ C\|\omega\|_{\dot{H}^{-1}}^2 + R^{\delta_1,\delta_2} +o(1),
\end{equation*}
where $R^{\delta_1,\delta_2}$ is a $\R_+$-valued random variable whose expectation is $o(1)$ as $\delta_1, \delta_2\rightarrow 0$.

Now we turn to estimate the term $J_1$, especially the nonlinear part. Recalling $u=K*\omega=\nabla^\perp G*\omega$, we have
\begin{equation*}
	\big\< G*\omega,u\cdot \nabla \omega^2 \big\>
	=-\< u\cdot \nabla G*\omega, \omega^2  \>=0.
\end{equation*}
Fix $\bar{\alpha}\in \big(\alpha, 1-\frac{1}{p\wedge 2} \big)$, by Lemma \ref{product} with $a=2\bar{\alpha}, b=1-\bar{\alpha}$, we get
\begin{equation}\label{eq-nonlinear-part}
\aligned
	\big|\big\< G*\omega,u^1\cdot \nabla \omega \big\> \big|
	&= \big|\big\<u^1\cdot \nabla G*\omega, \omega \big\> \big|  \\
	&\leq \|u^1\cdot \nabla G*\omega\|_{\dot{H}^{\bar{\alpha}}}
	\|\omega\|_{\dot{H}^{-\bar{\alpha}}}
	\\
	&\lesssim \|u^1\|_{\dot{H}^{2\bar{\alpha}}}
	\|\nabla G*\omega\|_{\dot{H}^{1-\bar{\alpha}}}
	\|\omega\|_{\dot{H}^{-\bar{\alpha}}}  \\
	&\lesssim \|\omega^1\|_{\dot{H}^{2\bar{\alpha}-1}}
	\|\omega\|_{\dot{H}^{-\bar{\alpha}}}^2.
\endaligned
\end{equation}
By Sobolev embedding $L^{p\wedge 2}\hookrightarrow \dot{H}^{1-\frac{2}{p\wedge 2}}$,
and interpolation $\dot{H}^{-1}\cap \dot{H}^{1-\frac{2}{p\wedge 2}}\subset \dot{H}^{2\bar{\alpha}-1}$, we have
\begin{align*}
	\|\omega_t^1\|_{\dot{H}^{2\bar{\alpha}-1}}
	\leq \|\omega_t^1\|_{\dot{H}^{-1}}
	+ \|\omega_t^1\|_{L^{p\wedge 2}}
	\leq  C\Big(\|\omega^1\|_{C([0,t]; \dot{H}^{-1})}\Big), \quad \forall t\in [0,T],~\P\text{-a.s.}
\end{align*}
where the last inequality follows from the pathwise bounds \eqref{case1} and \eqref{case2}, and the symbol $C\Big(\|\omega^1\|_{C([0,t]; \dot{H}^{-1})}\Big)$ is a $\R_+$-valued random variable depending on $\|\omega^1\|_{C([0,t]; \dot{H}^{-1})}$ (linearly or nonlinearly)  and we omit the dependence on time and initial data. This notation will be frequently used  and we do not specifically mention it again. Therefore by interpolation $\dot{H}^{-1}\cap \dot{H}^{-\alpha}\subset \dot{H}^{-\bar{\alpha}}$ and $\big|\big\< G*\omega,\partial_1\theta \big\>\big|\leq \|\omega\|_{\dot{H}^{-1}}\, \|\theta\|_{L^2}\leq \|\omega\|_{\dot{H}^{-1}}^2 + \|\theta\|_{L^2}^2$, we have the  pathwise estimate
\begin{align*}
	J_1
	&\leq C\Big(\|\omega^1\|_{C([0,t]; \dot{H}^{-1})}\Big) \|\omega\|_{\dot{H}^{-\alpha}}^{2(1-\bar\alpha)/(1-\alpha)} \|\omega\|_{\dot{H}^{-1}}^{2(\bar\alpha -\alpha)/(1-\alpha)}
	+ \|\omega\|_{\dot{H}^{-1}}^2 + \|\theta\|_{L^2}^2 \\
	&\leq \frac{c_{\alpha}}{8} \|\omega\|_{\dot{H}^{-\alpha}}^2
	+ C\Big(\|\omega^1\|_{C([0,t]; \dot{H}^{-1})}\Big)
	\|\omega\|_{\dot{H}^{-1}}^2 + \|\theta\|_{L^2}^2.
\end{align*}

Summing up the above estimates, we get
\begin{equation}\label{diff-w-H^-1}
	\begin{split}
		\d \, \|\omega\|_{\dot{H}^{-1}}^2 + \frac{c_{\alpha}}{4}\|\omega\|_{\dot{H}^{-\alpha}}^2\,\d t
		&\leq  C\Big(\|\omega^1\|_{C([0,t]; \dot{H}^{-1})}\Big)
		\|\omega\|_{\dot{H}^{-1}}^2\,\d t
		+ \|\theta\|_{L^2}^2\,\d t \\
		&\quad + \sum_{k}\d (M_k^1+M_k^2)
		 +R^{\delta_1,\delta_2}\,\d t + o(1).
	\end{split}
\end{equation}

\textbf{Step 2: \boldmath{$\d \|\theta^{1}-\theta^{2}\|_{\dot{H}^{-1}}^2 $}.}
Note that the term $\|\theta\|_{L^2}^2$ appears on the right-hand side of \eqref{diff-w-H^-1}, now we need to compute $\|\theta\|_{\dot{H}^{-1}}^2$ to close the estimate. The difference $\theta$ satisfies
\begin{align*}
	\partial_t\theta + u^1\cdot \nabla\theta
	+ u\cdot \nabla\theta^2 =\Delta \theta.
\end{align*}
We have the equation
\begin{align*}
	\frac{\d}{\d t} \< \theta,G*\theta \> + 2\|\theta\|_{L^2}^2 &= -2\< u^1\cdot \nabla\theta, G*\theta\> -2\< u\cdot \nabla\theta^2, G*\theta \> \\
	&=:J_4 + J_5.
\end{align*}
Similarly, we have the estimate for $J_4$:
\begin{align*}
	J_4&\lesssim  \|u^1\|_{\dot{H}^{2\alpha}}
	\|\nabla G*\theta\|_{\dot{H}^{1-\alpha}}
	\|\theta\|_{\dot{H}^{-\alpha}}
	\leq \|\omega^1\|_{\dot{H}^{2\alpha-1}}\|\theta\|_{\dot{H}^{-\alpha}}^2 \\
	&\leq \frac{1}{8}\|\theta\|_{L^2}^2 + C\Big(\|\omega^1\|_{C([0,t]; \dot{H}^{-1})}\Big)
	\|\theta\|_{\dot{H}^{-1}}^2.
\end{align*}
As for the term $J_5$, applying Lemma \ref{product} with $a=b=1-\alpha$ and interpolation, we obtain
\begin{align*}
	J_5& \lesssim|\< u\cdot \nabla G*\theta,\theta^2 \>|
\leq  \|u\cdot \nabla G*\theta\|_{\dot{H}^{1-2\alpha}}
	\|\theta^2\|_{\dot{H}^{2\alpha-1}}
	\\
	&\lesssim \|u\|_{\dot{H}^{1-\alpha}}
	\|\nabla G*\theta\|_{\dot{H}^{1-\alpha}}
	\|\theta^2\|_{\dot{H}^{2\alpha-1}} \\
	&\lesssim \|\omega\|_{\dot{H}^{-\alpha}}
	\|\theta\|_{\dot{H}^{-\alpha}}
	\| \theta^2\|_{\dot{H}^{2\alpha-1}};
\end{align*}
by Young's inequality,
\begin{align*}
	J_5&\leq \frac{c_\alpha}{8} \|\omega\|_{\dot{H}^{-\alpha}}^2
	+ C\| \theta^2\|_{\dot{H}^{2\alpha-1}}^2 \|\theta\|_{\dot{H}^{-1}}^{2\alpha} \|\theta\|_{L^2}^{2(1-\alpha)}    \\
	&\leq \frac{c_\alpha}{8} \|\omega\|_{\dot{H}^{-\alpha}}^2 + \frac{1}{8}\|\theta\|_{L^2}^2 +  C\| \theta^2\|_{\dot{H}^{2\alpha-1}}^{2/\alpha} \|\theta\|_{\dot{H}^{-1}}^2.
\end{align*}
Recalling the uniform bound $\|\theta^2\|_{L^2}^2\leq \|\theta_0^{\delta_2}\|_{L^2}^2=\|\theta_0\|_{L^2}^2+o(1)$ and the pathwise bound \eqref{theta-delta-H^-1-bdd1}, by interpolation, it is obvious that $\|\theta^2\|_{\dot{H}^{2\alpha-1}}^{2/\alpha}\leq C\Big(\|\omega^2\|_{C([0,t]; \dot{H}^{-1})}\Big)$, thus combining the above estimates on $J_4$ and $J_5$,
we obtain
\begin{equation}\label{diff-theta-H^-1}
	\frac{\d}{\d t}  \|\theta\|_{\dot{H}^{-1}}^2 + \|\theta\|_{L^2}^2
	\leq
	\frac{c_\alpha}{8} \|\omega\|_{\dot{H}^{-\alpha}}^2
	+ C\Big(\|\omega^2\|_{C([0,t]; \dot{H}^{-1})}\Big)
	\|\theta\|_{\dot{H}^{-1}}^2,\quad \P\text{-a.s.}
\end{equation}

\textbf{Step 3: Estimate of \boldmath{$\|\omega^{1}-\omega^{2}\|_{C_T\dot{H}^{-1}}^2+\|\theta^{1}-\theta^{2}\|_{C_T\dot{H}^{-1}}^2$}.}
Putting \eqref{diff-w-H^-1}, \eqref{diff-theta-H^-1} together and writing them in integral form, we obtain $\P\text{-a.s.}, \forall\, t\in [0,T]$,
\begin{equation}\label{diff-H^-1}
	\begin{split}
		&\|\omega_t\|_{\dot{H}^{-1}}^2
		+ \|\theta_t\|_{\dot{H}^{-1}}^2
		+\int_{0}^{t}\|\omega_s\|_{\dot{H}^{-\alpha}}^2\,\d s
		+\int_{0}^{t}\|\theta_s\|_{L^2}^2  \,\d s \\
		&\lesssim \|\bar{\omega}_0\|_{\dot{H}^{-1}}^2
		+ \|\bar{\theta}_0\|_{\dot{H}^{-1}}^2
		+ \sum_{k} \big(M_k^{1,t}+M_k^{2,t} \big)
	    + TR^{\delta_1,\delta_2} + o(1)  \\
		&\quad
		+ \int_{0}^{t}C\Big(\|\omega^1\|_{C([0,s]; \dot{H}^{-1})}\Big)
		\|\omega_s\|_{\dot{H}^{-1}}^2 \,\d s
		+ \int_{0}^{t}C\Big(\|\omega^2\|_{C([0,s]; \dot{H}^{-1})}\Big)
		\|\theta_s\|_{\dot{H}^{-1}}^2 \,\d s,
	\end{split}
\end{equation}
where $(\bar{\omega}_0,\bar{\theta}_0)= \big(\omega_0^{\delta_1}- \omega_0^{\delta_2}, \theta_0^{\delta_1}- \theta_0^{\delta_2} \big)$.

Now for every $M>0$, we introduce a stopping time
\begin{equation*}
	\tau_M:=\inf\big\{t\geq0: \|\omega_t^1\|_{\dot{H}^{-1}}\vee\|\omega_t^2\|_{\dot{H}^{-1}} >M \big\}\wedge T.
\end{equation*}
The stopping time $\tau_M<\infty$ is well defined since $\omega^1,\omega^2\in C_T\dot{H}^{-1}, ~\P\text{-a.s.}$
Then we have
\begin{equation*}
	C\Big(\|\omega^1\|_{C([0,t]; \dot{H}^{-1})}\Big)
	\vee
	C\Big(\|\omega^2\|_{C([0,t]; \dot{H}^{-1})}\Big) \leq C(M)< +\infty, \quad \text{ on } \{t\leq \tau_M\}.
\end{equation*}
As a consequence, replacing $t$ with $t\wedge \tau_M$ and taking expectation in \eqref{diff-H^-1}, we arrive at
\begin{align*}
	&\E\big[\|\omega_{t\wedge \tau_M}\|_{\dot{H}^{-1}}^2\big]
	+ \E\big[\|\theta_{t\wedge \tau_M}\|_{\dot{H}^{-1}}^2\big]
	+\E\int_{0}^{t\wedge \tau_M}\|\omega_s\|_{\dot{H}^{-\alpha}}^2\,\d s \\
	&\lesssim
	\|\bar{\omega}_0\|_{\dot{H}^{-1}}^2
	+ \|\bar{\theta}_0\|_{\dot{H}^{-1}}^2
	+ C(M) \int_{0}^{t}\Big(\E\big[ \|\omega_{s\wedge \tau_M}\|_{\dot{H}^{-1}}^2 \big]
	+ \E\big[ \|\theta_{s\wedge \tau_M}\|_{\dot{H}^{-1}}^2 \big]\Big)\,\d s + o(1).
\end{align*}
Thus by Gr\"onwall's inequality, we obtain
\begin{equation}\label{suptE}
	\begin{split}
		&\sup_{t\in [0,T]} \Big(\E\big[\|\omega_{t\wedge \tau_M}\|_{\dot{H}^{-1}}^2\big]
		+ \E\big[\|\theta_{t\wedge \tau_M}\|_{\dot{H}^{-1}}^2\big]\Big)
		+ \E\int_{0}^{ \tau_M}\|\omega_s\|_{\dot{H}^{-\alpha}}^2\,\d s\\
		&\lesssim_{T,M} \|\bar{\omega}_0\|_{\dot{H}^{-1}}^2
		+ \|\bar{\theta}_0\|_{\dot{H}^{-1}}^2 + o(1).
	\end{split}
\end{equation}

To interchange the order of the supremum over time and  expectation, from \eqref{diff-H^-1}, we have
\begin{equation}\label{Esupt}
	\begin{split}
		&\E\bigg[\sup_{t\in [0,T]}\|\omega_{t\wedge \tau_M}\|_{\dot{H}^{-1}}^2\bigg]
		+ \E\bigg[\sup_{t\in [0,T]}\|\theta_{t\wedge \tau_M}\|_{\dot{H}^{-1}}^2 \bigg]
		+\E\int_{0}^{\tau_M}\|\omega_s\|_{\dot{H}^{-\alpha}}^2\,\d s
		\\
		&\lesssim \|\bar{\omega}_0\|_{\dot{H}^{-1}}^2
		+ \|\bar{\theta}_0\|_{\dot{H}^{-1}}^2
		+ \sum_{i=1,2}\E\bigg[\Big|\sup_{t\in [0,T]}\sum_{k}M_k^{i,t\wedge\tau_M}\Big|\bigg]
		\\
		&\quad
		+ C(M) \int_{0}^{T}\Big(\E\big[ \|\omega_{t\wedge \tau_M}\|_{\dot{H}^{-1}}^2 \big]
		+ \E\big[ \|\theta_{t\wedge \tau_M}\|_{\dot{H}^{-1}}^2 \big]\Big)\,\d t+ o(1).
	\end{split}
\end{equation}
By Burkholder-Davis-Gundy inequality and proceeding as the computations below \eqref{H^-1-mid}, we have for every $\epsilon>0$ small,
\begin{equation}\label{BDG1}
	\begin{split}
		\E\bigg[\Big|\sup_{t\in [0,T]}\sum_{k}M_k^{1,t\wedge\tau_M}\Big|\bigg]
		&= 2\E\bigg[\Big|\sup_{t\in [0,T]}\sum_{k}\int_{0}^{t\wedge \tau_M}\big\< G*\omega_s, \sigma_k^1 \cdot \nabla\omega_s\big\>\,\d W_s^k\Big|\bigg]   \\
		& \leq \epsilon\, \E\bigg[\sup_{t\in [0,T]} \|\omega_{t\wedge \tau_M}\|_{\dot{H}^{-1}}^2 \bigg]
		+ C_{\epsilon} \E\int_{0}^{\tau_M}\|\omega_t\|_{\dot{H}^{-\alpha}}^2 \, \d t.
	\end{split}
\end{equation}

For the term $M_k^2$, we have
\begin{align*}
	\E\bigg[\Big|\sup_{t\in [0,T]}\sum_{k}M_k^{2,t\wedge\tau_M}\Big|\bigg]
	&= 2\E\bigg[\Big|\sup_{t\in [0,T]}\sum_{k}\int_{0}^{t\wedge \tau_M}\big\< G*\omega_s, \bar{\sigma}_k \cdot \nabla\omega_s^2\big\>\,\d W_s^k\Big|\bigg]  \\
	&\lesssim \E \Bigg[\bigg(\int_{0}^{T} \sum_k\big|\big\< \omega_t^2, \bar{\sigma}_k \cdot \nabla G*\omega_t\big\>\big|^2\,\d t \bigg)^{\frac{1}{2}}\Bigg].
\end{align*}
Recall that $\sum_k\bar{\sigma}_k\otimes \bar{\sigma}_k=Q^{\delta_2}-Q^{\delta_1}$ and
\begin{equation*}
	\|Q^{\delta_2}-Q^{\delta_1}\|_{L^{\infty}}
	\leq \big\|\widehat{Q^{\delta_1}}-\widehat{Q^{\delta_2}}\big\|_{L^1}
	\lesssim \int_{\frac{1}{\delta_1}\leq |\xi|\leq \frac{1}{\delta_2} }\< \xi\>^{-2-2\alpha}\,\d \xi
	\lesssim c_2-c_1,
\end{equation*}
then a direct computation shows that
\begin{align*}
	\sum_k\big|\big\< \omega^2, \bar{\sigma}_k \cdot \nabla G*\omega\big\>\big|^2
	&= \sum_k\iint \bar{\sigma}_k(x)\cdot\nabla G*\omega(x)\omega^2(x) \bar{\sigma}_k(y)\cdot\nabla G*\omega(y)\omega^2(y)\,\d x\d y\ \\
	&=\Big\<\big(Q^{\delta_1}-Q^{\delta_2}\big)*\big((\nabla G*\omega) \omega^2\big), (\nabla G*\omega) \omega^2\Big\> \\
	&\leq \|Q^{\delta_1}-Q^{\delta_2}\|_{L^{\infty}}
	\|(\nabla G*\omega) \omega^2\|_{L^1}^2   \\
	&\leq \|Q^{\delta_1}-Q^{\delta_2}\|_{L^{\infty}} \|\omega^2\|_{L^2}^2\|\omega\|_{\dot{H}^{-1}}^2\\
	&\lesssim |c_1-c_2| \big(\|\omega_0^{\delta_2}\|_{L^2}^2+1\big) \|\omega\|_{\dot{H}^{-1}}^2.
\end{align*}
Thus by \eqref{c-delta},  Proposition \ref{key-bdd} and the fact that $(\omega_0^{\delta})_{\delta}$ is bounded in $L^2$, we get the estimate
\begin{equation}\label{BDG2}
	\E\bigg[\Big|\sup_{t\in [0,T]}\sum_{k}M_k^{2,t\wedge\tau_M}\Big|\bigg] \lesssim \delta_1^{\alpha}\, \big(\|\omega_0^{\delta_2}\|_{L^2}+1\big)
	\bigg(\int_{0}^{T}\E\Big[\|\omega_t\|_{\dot{H}^{-1}}^2\Big]\,\d t\bigg)^{\frac{1}{2}}= o(1).
\end{equation}

Substituting the bounds \eqref{BDG1} with $\epsilon$ sufficiently small and \eqref{BDG2} into \eqref{Esupt}, we obtain
\begin{align*}
	&\E\bigg[\sup_{t\in [0,T]}\|\omega_{t\wedge \tau_M}\|_{\dot{H}^{-1}}^2\bigg]
	+ \E\bigg[\sup_{t\in [0,T]}\|\theta_{t\wedge \tau_M}\|_{\dot{H}^{-1}}^2 \bigg] \\
	&\lesssim
	\|\bar{\omega}_0\|_{\dot{H}^{-1}}^2
    +\|\bar{\theta}_0\|_{\dot{H}^{-1}}^2
    + C(M)\int_{0}^{T}\Big(\E\big[ \|\omega_{t\wedge \tau_M}\|_{\dot{H}^{-1}}^2 \big]
	+ \E\big[ \|\theta_{t\wedge \tau_M}\|_{\dot{H}^{-1}}^2 \big]\Big)\,\d t
	+\E\int_{0}^{\tau_M}\|\omega_t\|_{\dot{H}^{-\alpha}}^2 \, \d t,
\end{align*}
which, combined with \eqref{suptE}, gives the bound
\begin{equation}\label{ESupt}
	\begin{split}
		\E\bigg[\sup_{t\in [0,T]}\|\omega_{t\wedge \tau_M}\|_{\dot{H}^{-1}}^2\bigg]
		+ \E\bigg[\sup_{t\in [0,T]}\|\theta_{t\wedge \tau_M}\|_{\dot{H}^{-1}}^2 \bigg]
		\lesssim_{T,M} \|\bar{\omega}_0\|_{\dot{H}^{-1}}^2
		+ \|\bar{\theta}_0\|_{\dot{H}^{-1}}^2 + o(1).
	\end{split}
\end{equation}

\textbf{Step 4: Conclusion.}
Next, we show how to prove the desired results from the estimate \eqref{ESupt}. We only prove it for $\omega=\omega^{\delta_1}-\omega^{\delta_2}$, since the case of $\theta=\theta^{\delta_1}-\theta^{\delta_2}$ is exactly the same. For every $\epsilon,\, M>0$, note that $\{\tau_M<T\}=\{\|\omega^1\|_{C_T\dot{H}^{-1}}\vee \|\omega^2\|_{C_T\dot{H}^{-1}}>M\}$, we have
\begin{align*}
	\P\Big( \|\omega\|_{C_{T}\dot{H}^{-1}}>\epsilon \Big)
	&= \P\Big( \|\omega\|_{C_{T}\dot{H}^{-1}}>\epsilon,\tau_M= T \Big)
	+ \P\Big( \|\omega\|_{C_{T}\dot{H}^{-1}}>\epsilon,\tau_M< T \Big) \\
	&\leq \P\Big(\sup_{t\in [0,T]} \|\omega_{t\wedge\tau_M}\|_{\dot{H}^{-1}}>\epsilon,\tau_M=T \Big)
	+ \P\big( \tau_M< T \big) \\
	&\leq \P\Big(\sup_{t\in [0,T]} \|\omega_{t\wedge\tau_M}\|_{\dot{H}^{-1}}>\epsilon\Big)
	+ \P\Big( \|\omega^1\|_{C_T\dot{H}^{-1}}\vee \|\omega^2\|_{C_T\dot{H}^{-1}}>M \Big).
\end{align*}
On the one hand, by \eqref{ESupt},
\begin{align*}
	\P\Big(\sup_{t\in [0,T]} \|\omega_{t\wedge\tau_M}\|_{\dot{H}^{-1}}>\epsilon\Big)
	&\leq \frac1{\epsilon^2} \E\bigg[\sup_{t\in [0,T]}\|\omega_{t\wedge \tau_M}\|_{\dot{H}^{-1}}^2\bigg]  \\
	&\lesssim \frac{C(M)}{\epsilon^2} \Big(\|\omega_0^{\delta_1}-\omega_0^{\delta_2}\|_{\dot{H}^{-1}}^2
	+ \|\theta_0^{\delta_1}-\theta_0^{\delta_2}\|_{\dot{H}^{-1}}^2 + o(1)\Big).
\end{align*}
On the other hand, from Proposition \ref{key-bdd},
\begin{align*}
	\P\Big( \|\omega^1\|_{C_T\dot{H}^{-1}}\vee \|\omega^2\|_{C_T\dot{H}^{-1}}>M \Big)
	&\leq
	\P\big( \|\omega^1\|_{C_T\dot{H}^{-1}}>M \big)
	+\P\big( \|\omega^2\|_{C_T\dot{H}^{-1}}>M \big)  \\
	&\leq \frac1{M^2} \E\Big[\|\omega^1\|_{C_T\dot{H}^{-1}}^2\Big]
	+ \frac1{M^2} \E\Big[\|\omega^2\|_{C_T\dot{H}^{-1}}^2\Big] \\
	&\leq \frac1{M^2} \sup_{\delta}\E\Big[\|\omega^{\delta}\|_{C_T\dot{H}^{-1}}^2\Big] \lesssim \frac1{M^2}.
\end{align*}
Putting them together, we obtain, for all $\epsilon,M>0$,
\begin{align*}
	\P\Big(\sup_{t\in [0,T]} \|\omega^{\delta_1}-\omega^{\delta_2}\|_{\dot{H}^{-1}}>\epsilon\Big) \leq \frac{C(M)}{\epsilon^2} \Big(\|\omega_0^{\delta_1}-\omega_0^{\delta_2}\|_{\dot{H}^{-1}}^2
	+ \|\theta_0^{\delta_1}-\theta_0^{\delta_2}\|_{\dot{H}^{-1}}^2 + o(1)\Big) + \frac{C}{M^2};
\end{align*}
note that all constants are independent of $\delta_1,\delta_2$, thus first letting $\delta_1,\delta_2\to 0$, and then $M\to\infty$ complete the proof.
\end{proof}

\begin{rmk}\label{tech-error-noise}
	The technical constraint $p\geq 2$ enables us to easily control the error \eqref{err-noise} related to different noises; indeed, it would be difficult to bound \eqref{tech-bdd} without this constraint. It seems that we should estimate \eqref{err-noise} in frequency space, that is, estimate its Fourier modes and use the $L_T^2\dot{H}^{-\alpha}$ bound to control the terms in \eqref{tech-bdd}, which is extremely sophisticated.
\end{rmk}

Proposition \ref{converge-H^-1} and the well known Borel-Cantelli Lemma (see e.g. \cite[Section 2.3]{Dur}) yield the following convergence:
\begin{cor}\label{cor1}
	There exist a subsequence $(\omega^{\delta_n},\theta^{\delta_n})_{n\in \N}$ and a limit $(\omega,\theta)$ with trajectories in $C_T\dot{H}^{-1}\times C_T\dot{H}^{-1}$  such that
	\begin{equation*}
		(\omega^{\delta_n},\theta^{\delta_n})\rightarrow (\omega,\theta) \quad\text{ in }C_T\dot{H}^{-1}\times C_T\dot{H}^{-1} \text{ as }n\rightarrow\infty, \quad\P\text{-a.s.}
	\end{equation*}
\end{cor}

Once we have $\P\text{-a.s.}$ convergence in $C_T\dot{H}^{-1}$, combined with the pathwise bounds \eqref{case1}--\eqref{case3}, the  $\P\text{-a.s.}$ convergence holds in some stronger topologies. Note that we do not use the condition $p\geq 2$ explicitly in the following two results, which is crucial for the proof of Theorem \ref{main-thm}, see  the next section for details.

\begin{cor}\label{cor2}
    We have $\omega\in L_T^{\infty}\dot{H}^{1-\frac{2}{p\wedge 2}}$ and for all $s\in \big[-1,1-\frac{2}{p\wedge 2}\big)$,
	\begin{equation*}
		\omega^{\delta_n}\rightarrow \omega\quad \text{in } C_T\dot{H}^{s} \text{ as }n\rightarrow\infty, \quad\P\text{-a.s.}
	\end{equation*}
\end{cor}
\begin{proof}
	From the pathwise convergence $\omega^{\delta_n}\rightarrow \omega \text{ in } C_T\dot{H}^{-1}$, one deduces that
	\begin{equation*}
		\P\Big(\sup_{n}\|\omega^{\delta_n}\|_{C_T\dot{H}^{-1}}<\infty\Big)=1,
	\end{equation*}
which combined with the pathwise bounds \eqref{case1}--\eqref{case3}, yields
	\begin{equation*}
	    \P\Big(\sup_{n}\|\omega^{\delta_n}\|_{L_T^{\infty}L^{p\wedge 2}}<\infty\Big)=1;
	\end{equation*}
that is, $\P$-a.s., $\{\omega^{\delta_n}\}_n$ is a bounded set in $L_T^{\infty}L^{p\wedge 2}\subset L_T^{\infty}\dot{H}^{1-\frac{2}{p\wedge 2}}$. Then by property of weak$\ast$\,-convergence, the limit $\omega$ also belongs to $L_T^{\infty}L^{p\wedge 2}\subset L_T^{\infty}\dot{H}^{1-\frac{2}{p\wedge 2}} ,~\P\text{-a.s.}$ The following continuous embedding is obvious by interpolation between Sobolev spaces:
	\begin{equation*}
		L_T^{\infty}\dot{H}^{1-\frac{2}{p\wedge 2}}\cap C_T\dot{H}^{-1}\hookrightarrow C_T\dot{H}^s, \quad \forall s\in \Big[-1,1-\frac{2}{p\wedge 2}\Big).
	\end{equation*}
Hence we have $\omega^{\delta_n},\omega\in C_T\dot{H}^{s}$ for every $s\in \big[-1,1-\frac{2}{p\wedge 2}\big)$. Choosing $\lambda\in (0,1)$ such that $s=(1-\lambda)(-1)+\lambda (1-\frac{2}{p\wedge 2})$, we have
    \begin{equation*}
    	\|\omega^{\delta_n}-\omega\|_{C_T\dot{H}^{s}}\leq \|\omega^{\delta_n}-\omega\|_{C_T\dot{H}^{-1}}^{1-\lambda}
    	\|\omega^{\delta_n}-\omega\|_{L_T^\infty L^{p\wedge 2}}^{\lambda}\rightarrow 0 \quad \text{ as }n\rightarrow \infty,\quad \P\text{-a.s.} \qedhere
    \end{equation*}
\end{proof}
This corollary, combined with the theory of heat kernels associated with singular drift terms, yields convergence of $\theta^{\delta_n}$ in a much stronger topology. The proof relies on Corollary \ref{cor2} and estimates similar to \textbf{Cases 1}--\textbf{3} in Section \ref{appr}, and the details are left in Appendix B.

\begin{cor}\label{cor3}
	We have $\theta\in{C_TL^2}\cap L_T^2\dot{H}^1\cap L_T^1(\dot{W}^{1,1}\cap\dot{W}^{1,p})$ and $\P$-a.s.,
	\begin{equation*}
		\theta^{\delta_n}\rightarrow \theta \quad \text{in }{C_TL^2}\cap L_T^2\dot{H}^1\cap L_T^1(\dot{W}^{1,1}\cap\dot{W}^{1,p})\quad \text{as } n\rightarrow \infty .
	\end{equation*}
\end{cor}

\begin{rmk}
	Corollary \ref{cor3} seems a little surprising at first glance. Indeed, the space-time integrability of the drift term $u$ is subcritical under the well known  Krylov-R\"ockner conditions \cite{KR05}, hence $\theta$, as the solution of transport-diffusion equation (with random coefficients), can be represented as the ``convolution'' of initial data and fundamental solution corresponding to the transport-diffusion equation. This fundamental solution shares similar properties with the standard heat kernel, such as the two sided Gaussian estimate and upper gradient estimate, see \cite[Theorem 1.1]{Chen} and \cite[Lemma 3.9]{Zhao} for details. Thus when the difference of two drift terms are sufficiently small in the subcritical regime, it is natural to believe that the corresponding solutions are sufficiently close.
\end{rmk}

\section{Proof of Theorem \ref{main-thm}}\label{proof}

With the above preparations, we can now prove the existence and uniqueness of strong solutions (in probabilistic sense) to \eqref{Boussinesq-Ito}.

\begin{proof}[Proof of Theorom \ref{main-thm}]
\textbf{Part 1.} We shall first focus on the case $p\geq 2$.
	
	\textbf{Existence of probabilistic strong solutions.}
	Our task now is to verify \eqref{path-bdd}--\eqref{ex-path-bdd}.
	From Lemma \ref{wellpoesd-reg-model}, Proposition \ref{key-bdd} and \eqref{w-Lp-mid}--\eqref{case3}, we have obtained the following pathwise bounds for the approximate sequence:
	\begin{equation}\label{1}
		\begin{split}
			\sup_{t\in [0,T]}\|\theta_t^{\delta_n}\|_{L^2}^2
			&+ \int_{0}^{T}\|\nabla\theta_t^{\delta_n}\|_{L^{2}}^2\,\d t\leq2\|\theta_0\|_{L^2}^2 + o(1),  \\
			\int_{0}^{T} \|\nabla\theta_t^{\delta_n}\|_{L^1\cap L^p}\,\d t
			& \leq C\Big(p, T,\|\omega_0\|_{L^1\cap L^p},~\|\theta_0\|_{L^1\cap L^2},~  \|\omega^{\delta_n}\|_{C_T\dot{H}^{-1}}\Big),  \\
			\|\omega^{\delta_n} \|_{L_T^{\infty}L^{p}} &\le \|\omega_0 \|_{L^p}+\int_{0}^{T} \|\nabla\theta_s^{\delta_n} \|_{L^{p}} \,\d s + o(1),
		\end{split}
	\end{equation}
and moreover,
	\begin{equation}\label{11}
		\E\bigg[\sup_{t\in [0,T]} \|\omega_t^{\delta_n}\|_{\dot{H}^{-1}}^2 \bigg]
		+\int_{0}^{T}\E\big[\|\omega_t^{\delta_n}\|_{\dot{H}^{-\alpha}}^2 \big]\, \d t
		\lesssim_{\alpha, T} \|\omega_0\|_{\dot{H}^{-1}}^2 +  \|\theta_0\|_{L^2}^2 + o(1).
	\end{equation}
Thanks to the strong convergence obtained in Corollaries \ref{cor1}--\ref{cor3}, letting $n\rightarrow\infty$ in \eqref{1} gives us the corresponding pathwise bounds for the limit $(\omega,\theta)$. As for the energy bound \eqref{11}, we can argue as \cite[(6.5)]{CogMau}, which relies on the lower semi-continuity of $L_T^2\dot{H}^{-\alpha}$ norm in $L_T^2\dot{H}^{-1}$.
	
	It remains to check the identities \eqref{weak}. Since we have obtained stronger convergence than that in \cite[Section 6]{CogMau}, we omit the details and refer the reader to \cite[Section 6, \textbf{Step 3}]{CogMau}. \smallskip

	\textbf{Pathwise uniqueness.} Assume $(\omega^1,\theta^1),(\omega^2,\theta^2)$ are two solutions to \eqref{Boussinesq-Ito} with  the same initial data $(\bar{\omega}_0,\bar{\theta}_0)$ in $(\dot{H}^{-1}\cap L^1\cap L^p)\times (\dot{H}^{-1}\cap L^1\cap L^2)$, where $p\in (1,\infty)$. Then $\omega:=\omega^1-\omega^2,\, \theta:=\theta^1-\theta^2$ satisfy
	\begin{equation*}
		\left\{\aligned
		& \partial_t \omega + u^{1}\cdot \nabla\omega + u\cdot \nabla\omega^{2}
		+ \sum_{k}\sigma_k \cdot \nabla\omega\,\dot W^k
		=\partial_1\theta + \frac{\pi}{4\alpha}\Delta \omega,
	    \\
		& \partial_t\theta + u^{1}\cdot \nabla\theta + u\cdot \nabla\theta^{2} = \Delta \theta,
		\endaligned \right.
	\end{equation*}
where $u^i=K*\omega^i,i=1,2 \text{ and }u=K*\omega$.
Since $\theta^1,\theta^2$ both have trajectories in  $L_T^{\infty}(\dot{H}^{-1}\cap L^2) \cap L_T^2H^1$, we can directly compute $\frac{\d}{\d t}\<\theta,G*\theta \>$. Indeed, the bound $\|\omega^{i}\|_{L_T^{\infty}(L^1\cap L^p)}<\infty$, $\P\text{-a.s.}$ implies for some $r\in (2, \infty]$, $\|u^{i}\|_{L_T^{\infty}L^r}<\infty$, $\P\text{-a.s.}$ Hence by H\"older's inequality,  we have
    \begin{align*}
    	\int_{0}^{T}\|\partial_t\theta_t^{i}\|_{\dot{H}^{-1}}^2 	\,\d t
    	&\lesssim \int_{0}^{T}\|u_t^{i}\cdot\nabla\theta_t^{i}\|_{\dot{H}^{-1}}^2 	\,\d t
    	+  \int_{0}^{T}\|\Delta\theta_t^{i}\|_{\dot{H}^{-1}}^2 	\,\d t \\
    	&\lesssim  \int_{0}^{T}\|u_t^{i}\theta_t^{i}\|_{L^{2}}^2 \,\d t
    	+ \int_{0}^{T}\|\nabla\theta_t^{i}\|_{L^{2}}^2\,\d t\\
    	&\leq \int_{0}^{T}\|u_t^{i}\|_{L^{r}}^2 \|\theta_t^{i}\|_{L^{q}}^2\,\d t + \|\theta_0^{i}\|_{L^2}^2  \\
    	&\leq \|u^{i}\|_{L_T^{\infty}L^r}^2 \int_{0}^{T} \|\theta_t^{i}\|_{L^{q}}^2\,\d t + \|\theta_0^{i}\|_{L^2}^2,
    \end{align*}
where $q\in [2,\infty)$ satisfy $1/q+1/r=1/2$. Then by Sobolev embedding and interpolation, we get
    \begin{equation*}
    	\int_{0}^{T} \|\theta_t^{i}\|_{L^{q}}^2\,\d t
    	\leq \int_{0}^{T} \|\theta_t^{i}\|_{L^{2}}^2\,\d t
    	+ \int_{0}^{T}\|\nabla\theta_t^{i}\|_{L^2}^2 \,\d t
    	\leq C\big(T,\|\theta_0^{i}\|_{L^2}^2\big),
    \end{equation*}
which leads to the pathwise bound
    \begin{equation*}
    \int_{0}^{T}\|\partial_t\theta_t^{i}\|_{\dot{H}^{-1}}^2 \,\d t
    \lesssim_{T,\theta_0^i}  \|u^{i}\|_{L_T^{\infty}L^r}^2 +1 \lesssim \|\omega^{i} \|_{L_T^{\infty} (L^1 \cap L^p)}^2 +1 <\infty,\quad\P\text{-a.s.}
    \end{equation*}
From the above bound, one can easily verify that $\frac{\d}{\d t}\<\theta,G*\theta \>=2\<G*\theta,\partial_t\theta\>$, thus we have
\begin{align*}
	\frac{\d}{\d t} \< \theta,G*\theta \> + 2\|\theta\|_{L^2}^2 &= -2\< u^1\cdot \nabla\theta, G*\theta\> -2\< u\cdot \nabla\theta^2, G*\theta \>.
\end{align*}
Proceeding as \eqref{diff-theta-H^-1}, we obtain
\begin{equation}\label{estim-theta}
	\frac{\d}{\d t}  \|\theta\|_{\dot{H}^{-1}}^2 + \|\theta\|_{L^2}^2
	\leq
	\frac{c_\alpha}{8} \|\omega\|_{\dot{H}^{-\alpha}}^2
	+ C\Big(\|\omega^2\|_{C([0,t]; \dot{H}^{-1})}\Big)
	\|\theta\|_{\dot{H}^{-1}}^2,\quad \P\text{-a.s.}
\end{equation}

To compute $\|\omega\|_{\dot{H}^{-1}}^2 =4\pi^2\< \omega,G*\omega \>$, we need to introduce a family of smooth kernel $(G^{\delta})_{\delta\in (0,1)}$ with $\widehat{G^{\delta}}(\xi)=\widehat{G}(\xi)\big(e^{-4\pi\delta|\xi|^2}-e^{-4\pi|\xi|^2/\delta}\big)$ for $\xi\in \R^2$, see \cite[Section 3.3]{CogMau} for the exact definition and more properties. Now applying It\^o's formula to $\<\omega, G^{\delta}*\omega\>$, similar to \eqref{Ito-reg-w-H^-1} and \eqref{sim2}, we get
\begin{equation}\label{eq-uniqueness-omega}
\aligned
	\d \< \omega, G^{\delta}*\omega \>
	&= -2\big\< G^{\delta}*\omega,u^1\cdot \nabla \omega \big\>\,\d t
	-2\big\< G^{\delta}*\omega,u\cdot \nabla \omega^2 \big\>\,\d t
	+ 2\big\< G^{\delta}*\omega, \partial_1\theta \big\> \,\d t  \\
	&\quad -2\sum_{k}\big\< G^{\delta}*\omega, \sigma_k\cdot \nabla\omega\big\>\,\d W^k
    +\big\< \tr\big[\big(Q(0)-Q\big)D^2G^{\delta}\big]*\omega,\omega\big\>\,\d t.
\endaligned
\end{equation}
Due to the choice of $G^\delta$, it is easy to deduce that
  $$\big|\big\< G^{\delta}*\omega, \partial_1\theta \big\> \big| \lesssim \|\theta\|_{L^2}^2 + \|\omega\|_{\dot{H}^{-1}}^2; $$
next, similar to \eqref{eq-nonlinear-part}, we have
\begin{align*}
	\big|\big\< G^{\delta}*\omega,u^1\cdot \nabla \omega \big\>\big|
	&\leq \frac{c_{\alpha}}{8} \|\omega\|_{\dot{H}^{-\alpha}}^2
	+ C\Big(\|\omega^1\|_{C([0,t]; \dot{H}^{-1})}\Big)
	\|\omega\|_{\dot{H}^{-1}}^2.
\end{align*}

For the term $\big\< G^{\delta}*\omega, u\cdot\nabla\omega^2\big\>$, integrating by parts gives us $\big\< G*\omega, u\cdot\nabla\omega^2\big\>=0$, then by Lemma \ref{product} with $a=2\alpha,b=1-\alpha$ we have
\begin{align*}
	\big|\big\< G^{\delta}*\omega, u\cdot\nabla\omega^2\big\>\big|
	&=\big|\big\< u\cdot\nabla(G-G^{\delta})*\omega, \omega^2\big\>\big|    \\
	&\lesssim \|u\|_{\dot{H}^{2\alpha}} \|\nabla(G-G^{\delta})*\omega\|_{\dot{H}^{1-\alpha}}
	\|\omega^2\|_{\dot{H}^{-\alpha}}  \\
	&\leq  \|\omega\|_{\dot{H}^{2\alpha-1}}
	 \|\nabla(G-G^{\delta})*\omega\|_{\dot{H}^{1-\alpha}}
	 \|\omega^2\|_{\dot{H}^{-\alpha}}.
\end{align*}
Note that
\begin{align*}
	\int_{0}^{T}\E\Big[\|\nabla(G-G^{\delta})*\omega_t\|_{\dot{H}^{1-\alpha}}
	\|\omega_t^2\|_{\dot{H}^{-\alpha}}\Big]\,\d t
	&\leq \bigg(\int_{0}^{T}\E\Big[\|\nabla(G-G^{\delta})*\omega_t\|_{\dot{H}^{1-\alpha}}^2\Big]\,\d t\bigg)^{1/2}\\
    &\quad \times \bigg(\int_{0}^{T}\E\Big[\|\omega_t^2\|_{\dot{H}^{-\alpha}}^2\Big]\,\d t\bigg)^{1/2}.
\end{align*}
Since the $L_{\Omega}^2L_T^2\dot{H}^{-\alpha}$ norms of $\omega^1,\omega^2$ are finite and $\widehat{G^{\delta}}$ converges to $\widehat{G}$ pointwise from below, then by the dominated convergence theorem, we get
\begin{equation*}
		\int_{0}^{T}\E\Big[\|\nabla(G-G^{\delta})*\omega_t\|_{\dot{H}^{1-\alpha}}
	\|\omega_t^2\|_{\dot{H}^{-\alpha}}\Big]\,\d t=o(1),\quad \delta\rightarrow 0.
\end{equation*}
Using Sobolev embedding $L^{p\wedge 2}\hookrightarrow \dot{H}^{1-\frac{2}{p\wedge 2}}$, and interpolation $\dot{H}^{-1}\cap \dot{H}^{1-\frac{2}{p\wedge 2}}\subset \dot{H}^{2\alpha-1}$ (due to $0<\alpha < 1- \frac1{p\wedge 2}$), we have
\begin{equation*}
		\|\omega_t\|_{\dot{H}^{2\alpha-1}}
		\leq \|\omega_t \|_{\dot{H}^{-1}}
		+ \|\omega_t\|_{L^{p\wedge 2}}
		\leq  C\Big(\|\omega\|_{C([0,t]; \dot{H}^{-1})}\Big), \quad \forall\, t\in [0,T],
\end{equation*}
where the last inequality follows from the pathwise bounds \eqref{ex-path-bdd}. Thus we have obtained the estimate
\begin{align*}
	\big|\big\< G^{\delta}*\omega, u\cdot\nabla\omega^2\big\>\big|
	\leq C\Big(\|\omega\|_{C([0,t]; \dot{H}^{-1})}\Big)
    \text{Rem}_1,
\end{align*}
where $\text{Rem}_1:=\|\nabla(G-G^{\delta})*\omega\|_{\dot{H}^{1-\alpha}}
\|\omega^2\|_{\dot{H}^{-\alpha}}$ satisfies
  $$\int_{0}^{T}\E\big[\text{Rem}_1\big]\,\d t=o(1)\quad \mbox{as } \delta\rightarrow 0.$$

Next, we split $\tr\big[\big(Q(0)-Q\big)D^2G^{\delta}\big]$ as follows:
\begin{align*}
	\tr\big[\big(Q(0)-Q\big)D^2G^{\delta}\big]
	&=\tr\big[\big(Q(0)-Q\big)D^2G\big]\varphi  \\
	&\quad-\tr\big[\big(Q(0)-Q\big)D^2(G-G^{\delta})\big]\varphi\\
	&\quad+\tr\big[\big(Q(0)-Q\big)D^2G^{\delta}\big](1-\varphi)\\
	&=:A+R_2+R_3,
\end{align*}
where $\varphi$ is a bump function. From \cite[Lemma 4.3,\,4.5]{CogMau}, there exist positive constants $c_{\alpha},C$ independent of $\delta$ and $\xi$ such that
\begin{align*}
	\widehat{A}(\xi)\leq -c_{\alpha}\langle \xi\rangle^{-2\alpha}+C\langle \xi\rangle^{-2},\quad
	\widehat{R_3}(\xi)\leq C\vert \xi\vert^{-2}, \quad \forall\,\xi\in \R^2.
\end{align*}
Moreover, by \cite[Lemma 3.1]{JiaoLuo},
\begin{align*}
	|\widehat{R}_2(\xi)|\leq C\<\xi\>^{-2\alpha},\quad |\widehat{R}_2(\xi)|\rightarrow 0\text{ as }\delta\rightarrow 0,\quad \forall\, \xi\in \R^2.
\end{align*}
Hence we have
\begin{equation*}
	\big\<\tr\big[\big(Q(0)-Q\big)D^2G^{\delta}\big]*\omega,\omega\big\>\leq -c_{\alpha}\|\omega\|_{\dot{H}^{-\alpha}}^2 +C\|\omega\|_{\dot{H}^{-1}}^2 + \text{Rem}_2,
\end{equation*}
where $\text{Rem}_2=\<R_2*\omega,\omega\>$ satisfies, by dominated convergence theorem,
  $$\int_{0}^{T}\E\big[\text{Rem}_2\big]\,\d t=o(1)\quad \mbox{as } \delta\rightarrow 0.$$
Substituting the above estimates in \eqref{eq-uniqueness-omega}, we obtain
  $$\aligned
  \d \< \omega, G^{\delta}*\omega \> &\le \Big(\|\theta \|_{L^2}^2 -\frac34 c_\alpha \|\omega \|_{\dot H^{-\alpha}}^2 \Big)\,\d t + C\Big(\|\omega^1\|_{C([0,t]; \dot{H}^{-1})}\Big) \|\omega\|_{\dot{H}^{-1}}^2 \, \d t \\
  &\quad + C\Big(\|\omega\|_{C([0,t]; \dot{H}^{-1})}\Big) (\text{Rem}_1 + \text{Rem}_2)\,\d t + \d M_t.
  \endaligned $$
Combining this inequality with \eqref{estim-theta}, we arrive at
  $$\aligned
  \d\big( \< \omega, G^{\delta}*\omega \> + \|\theta \|_{\dot H^{-1}}^2 \big) &\le -\frac{c_\alpha}2 \|\omega \|_{\dot H^{-\alpha}}^2 \,\d t + C\Big(\|\omega^1\|_{C([0,t]; \dot{H}^{-1})}\Big) \big( \|\omega\|_{\dot{H}^{-1}}^2 + \|\theta\|_{\dot{H}^{-1}}^2 \big) \, \d t \\
  &\quad + C\Big(\|\omega\|_{C([0,t]; \dot{H}^{-1})}\Big) (\text{Rem}_1 + \text{Rem}_2)\,\d t + \d M_t.
  \endaligned $$

Now for every $M>0$, we introduce stopping time $\tau_M$ as before:
\begin{equation*}
	\tau_M:=\inf\Big\{t\geq0: \|\omega_t^1\|_{\dot{H}^{-1}}\vee\|\omega_t^2\|_{\dot{H}^{-1}} >M\Big\}\wedge T.
\end{equation*}
Then arguing as in \eqref{suptE}, we arrive at
\begin{align*}
	&\sup_{t\in [0,T]}\E\big[\big\<\omega_{t\wedge \tau_M}, G^{\delta}*\omega_{t\wedge \tau_M}\big\>\big]
	+ \sup_{t\in [0,T]}\E\big[\|\theta_{t\wedge \tau_M}\|_{\dot{H}^{-1}}^2\big]
	+\E\int_{0}^{t\wedge \tau_M}\|\omega_s\|_{\dot{H}^{-\alpha}}^2\,\d s \\
	&\lesssim_{T,M}
	\|\omega_0\|_{\dot{H}^{-1}}^2
	+ \|\theta_0\|_{\dot{H}^{-1}}^2 +o(1).
\end{align*}
Next, the same argument as \eqref{ESupt} leads to
\begin{equation*}
	\begin{split}
		\E\bigg[\sup_{t\in [0,T]} \big\<\omega_{t\wedge \tau_M}, G^{\delta}*\omega_{t\wedge \tau_M} \big\>\bigg]
		+ \E\bigg[\sup_{t\in [0,T]}\|\theta_{t\wedge \tau_M}\|_{\dot{H}^{-1}}^2 \bigg]
		\lesssim_{T,M} \|\omega_0\|_{\dot{H}^{-1}}^2
		+ \|\theta_0\|_{\dot{H}^{-1}}^2 + o(1).
	\end{split}
\end{equation*}
Letting $\delta\rightarrow 0$, by the monotone convergence theorem, we get
\begin{equation}\label{Esupt-u}
	\begin{split}
		\E\bigg[\sup_{t\in [0,T]}\|\omega_{t\wedge \tau_M}\|_{\dot{H}^{-1}}^2\bigg]
		+ \E\bigg[\sup_{t\in [0,T]}\|\theta_{t\wedge \tau_M}\|_{\dot{H}^{-1}}^2 \bigg]
		\lesssim_{T,M} \|\omega_0\|_{\dot{H}^{-1}}^2
		+ \|\theta_0\|_{\dot{H}^{-1}}^2.
	\end{split}
\end{equation}
Since $\omega_0=\theta_0=0$,
this implies $\omega=\theta=0$ on $\{t\leq \tau_M\}$ for every $M>0$, then letting $M\rightarrow \infty$, we obtain $\P$-a.s. $\omega^1=\omega^2,\, \theta^1= \theta^2$. The proof of the case $p\geq 2$ is complete. \smallskip

\textbf{Part 2.} Now we assume $\omega_0\in \dot{H}^{-1}\cap L^1\cap L^p$, with $p\in (1,2)$ satisfying $\alpha<1-\frac{1}{p}$. Define as in Section 2 the approximate initial data $(\omega_0^{\delta},\theta_0^{\delta})$ which are smooth with compact support; we construct approximate solutions in a new way:
\begin{equation}\label{regular-model-u}
	\left\{\aligned
	& \partial_t \omega^{\delta} + u^{\delta}\cdot \nabla\omega^{\delta} + \sum_{k}\sigma_k \cdot \nabla\omega^{\delta} \,\dot{W}^k =\partial_1\theta^{\delta} + \frac{\pi}{4\alpha}\Delta \omega^{\delta},\\
	& \partial_t\theta^{\delta} + u^{\delta}\cdot \nabla\theta^{\delta} = \Delta \theta^{\delta},
    \quad u^{\delta}= K*\omega^{\delta},\\
	&\omega^{\delta}(0,\cdot) = \omega_0^{\delta}, ~\theta^{\delta}(0,\cdot) = \theta_0^{\delta}.
	\endaligned \right.
\end{equation}

Note that the system \eqref{regular-model-u} has smooth initial data but fixed rough noise, while the previous regularized system \eqref{regular-model} has both smooth initial data and smooth noise.
From the results about $p\geq 2$, one can easily verify that \eqref{regular-model-u} admits a unique solution $(\omega^{\delta},\theta^{\delta})$ satisfying bounds \eqref{path-bdd}--\eqref{weak} for every $\delta>0$.  Note that we did not explicitly exploit the condition $p\geq 2$ in the proof of Part 1. Repeating the \textbf{uniqueness} argument, we arrive at the estimate \eqref{Esupt-u} where $(\omega^i,\theta^i)$ is replaced by $(\omega^{\delta_i},\theta^{\delta_i}),\,i=1,2$.
The point is that the error term related to different noises in \eqref{err-noise} vanishes when estimating the difference $\|\omega^{\delta_1}-\omega^{\delta_2}\|_{C_T\dot{H}^{-1}}$.
Now arguing exactly as the proof of Proposition \ref{converge-H^-1},  the family $(\omega^{\delta}, \theta^{\delta})_{\delta}$ is pre-compact in $C_T\dot{H}^{-1}\times C_T\dot{H}^{-1},$  $\P\text{-a.s.}$  Since we do not use the condition $p\geq 2$ explicitly in Corollaries \ref{cor1}--\ref{cor3} and the above proof, it is obvious that the limit is the desired unique solution to \eqref{Boussinesq-Ito} by just repeating the procedure before.  The proof is finally complete.
\end{proof}

\begin{rmk}
	If $\omega$ is a probabilistic strong solution  and the initial data $\omega_0\in \dot{H}^{-1}\cap L^1\cap L^p$ with $p\geq 2$, it is also unnecessary to introduce the smooth kernel $G^{\delta}$. Indeed, define $H=\dot{H}^{-1},V=\dot{H}^{-1}\cap L^2$ and $V^*$ is the dual of $V$ with respect to $H$, the norms of which are as follows:
	\begin{align*}
		\|f\|_V^2:=\int_{\mathbb{R}^2}(|\xi|^{-2}\vee 1)\,|\hat{f}(\xi)|^2\,\d \xi,\quad
		\|f\|_{V^*}^2:=\int_{\mathbb{R}^2}(|\xi|^{-2}\wedge |\xi|^{-4})\,|\hat{f}(\xi)|^2\, \d \xi.
	\end{align*}
Then $(V,H,V^*)$ is a Gelfand triple, see \cite[Lemma 4.2]{BGM}. Since $\omega_0\in \dot{H}^{-1}\cap L^1\cap L^p\subset \dot{H}^{-1}\cap L^2$, we have $\omega\in L^2\big(\Omega, C_T\dot{H}^{-1}\big) \cap L^{\infty}\big(\Omega\times [0,T], L^2\big)$ and, for every $t\in [0,T]$,
\begin{align*}
	\omega_t& = \omega_0
	-\int_{0}^{t}\curl\div(u_s \otimes u_s)\, \d s
	+ \int_{0}^{t} \partial_1\theta_s  \, \d s
	+\frac{\pi}{4\alpha}\int_{0}^{t} \Delta \omega_s \, \d s
	-\sum_{k}\int_{0}^{t}\mathrm{div}(\sigma_k\omega_s) \, \d W_s^k  \\
	&=:\int_{0}^{t}Y(s) \,\d s +\int_{0}^{t}Z(s)\, \d W(s),
\end{align*}
where $Y=-\curl\div(u \otimes u)+\partial_1\theta+\frac{\pi}{4\alpha} \Delta \omega,\, Z=(\sigma_k\cdot\nabla\omega)_{k\in \N}$ and $W=(W^k)_{k\in \N}$ is the cylindrical Brownian motion on Hilbert space $U=\ell^2$. It is straightforward to verify that $Y\in L^2\big(\Omega\times [0,T],V^*\big)$ and $Z\in L^2\big(\Omega\times [0,T],L_2(U,H)\big)$, where $L_2(U,H)$ is the collection of all Hilbert-Schmidt operators from $U$ to $H$. Consequently, the stochastic Lions lemma \cite[Theorem 4.2.5]{LiuRoc} allows us to directly apply It\^o's formula to $\d \|\omega\|_{\dot{H}^{-1}}^2$.
\end{rmk}

\appendix

\section{Well-posedness of the regularized model \eqref{regular-model}}

In this section, we present the main ideas of the proof of  Lemma \ref{wellpoesd-reg-model}. To simplify notation, we omit the dependence on $\delta$ of $(\omega^\delta,\theta^\delta)$ and denote it by $(\omega,\theta)$ and assume the noise is smooth. Consider the following 2D stochastic Boussinesq system with full viscosity and smooth noise:

\begin{equation}\label{regular-model-A}
	\left\{\aligned
	& \partial_t \omega + u\cdot \nabla\omega + \sum_{k}    \L_k \omega \,\dot{W}^k =\partial_1\theta + \frac{1}{2}\sum_{k}\L_k^2\omega + \Delta \omega,\\
	& \partial_t\theta + u\cdot \nabla\theta = \Delta \theta, \\
	&\omega(0,\cdot) = \omega_0, ~\theta(0,\cdot) = \theta_0,
	\endaligned \right.
\end{equation}
where $u=K*\omega$, $\omega_0,\theta_0\in \dot{H}^{-1}\cap C_c^{\infty}$, $\L_k\omega=\sigma_k\cdot\nabla\omega$. For all $x,y\in \R^2$, we assume $Q(x-y)=\sum_{k}\sigma_k(x)\otimes\sigma_k(y)$  enjoys the same properties with $Q^\delta$ defined in \eqref{Q-delta} . In particular, $Q$ is even, $Q(0)=cI_2$ for some $c>0$ and $\widehat{Q}$ is compactly supported. Note that in this section, $Q$ is not the Kraichnan covariance mentioned in the main text.

In the absence of noise, the well-posedness of the deterministic 2D Boussinesq system with full viscosity is classical, which is similar to the 2D Navier-Stokes equations; see \cite{CanDi}. Therefore, we restrict ourselves to establishing a priori estimates for the stochastic system, the remaining part could be done by the classical compactness method for SPDEs.

\subsection*{Step 1: $\dot{H}^{-1}$ and $L^2$ estimate}

The $L^2$ energy estimate is quite straightforward, since divergence-free transport noise does not change $L^2$ bound, so we have
\begin{align*}
	\frac{\d}{\d t} \|\omega\|_{L^2} \leq \|\nabla\theta\|_{L^2},\quad \frac{\d}{\d t} \|\theta\|_{L^2}^2 + 2 \|\nabla\theta \|_{L^2}^2=0,\quad \P\text{-a.s.}
\end{align*}
We have the following pathwise bound
\begin{equation}\label{A-L^2-bdd}
	\|\theta_t\|_{L^2}^2 + 2\int_{0}^{t} \|\nabla\theta_s\|_{L^2}^2 \,\d s =\|\theta_0\|_{L^2}^2,\quad \|\omega_t\|_{L^2}\leq \|\omega_0\|_{L^2} +\sqrt{t}\|\theta_0\|_{L^2},\quad \forall t\in [0,T].
\end{equation}

Let $G$ be the Green function of $-\Delta$ on $\R^2$, namely $-\Delta G=\delta_0$, hence we have
\begin{equation*}
	\|\omega\|_{\dot{H}^{-1}}^2 = \int_{\mathbb{R}^2}|\xi|^{-2}\, |\widehat{\omega}(\xi)|^2 \,\d \xi = 4\pi^2\langle \omega,G*\omega \rangle= 4\pi^2 \|u\|_{L^2}^2.
\end{equation*}
By It\^o's formula, we have
\begin{equation}
	\begin{split} \label{A-Ito-H^-1}
		\d \langle \omega,G*\omega \rangle &=
		-2\langle G*\omega,u\cdot \nabla \omega \rangle \,\d t + 2\langle G*\omega, \partial_1\theta\rangle \,\d t
		+ 2\langle G*\omega,\Delta \omega \rangle \,\d t \\
		&\quad- 2\sum_{k}\langle G*\omega, \L_k\omega  \rangle \,\d W^k  \\
		&\quad+ \sum_{k} \langle G*\omega, \L_k^2\omega\rangle \,\d t
		+ \sum_{k}\langle G*\L_k\omega, \L_k\omega\rangle \,\d t.
	\end{split}
\end{equation}

By integration by parts, we know that $\langle G*\omega,u\cdot \nabla \omega \rangle=-\langle \nabla (G*\omega)\cdot u,  \omega \rangle=0$ and $\langle G*\omega,\Delta \omega \rangle = - \|\omega\|_{L^2}^2$; we also have $|\langle G*\omega,\partial_1\theta \rangle|\leq (2\pi)^{-1}\|\omega\|_{\dot{H}^{-1}}\,\|\partial_1\theta\|_{\dot{H}^{-1}}\leq \|u\|_{L^2}\|\theta\|_{L^2}$.

To compute the last two terms in \eqref{A-Ito-H^-1}, we need the following identity similar to \eqref{identity}:
\begin{equation}\label{A-identity}
	\L_k\omega= \sigma_k\cdot \nabla \omega= \curl \big((\sigma_k\cdot \nabla)u + (D\sigma_k)^T u\big),\quad\text{where } (D\sigma_k)_{ij}^T=\partial_i\sigma_k^j.
\end{equation}
Hence by Fourier transform and Parseval's identity, we have

\begin{align*}
	\sum_{k}\langle G*\L_k\omega, \L_k\omega\rangle
	&= \sum_{k}\big\langle G*\curl \big((\sigma_k\cdot \nabla)u + (D\sigma_k)^T u\big), \curl \big((\sigma_k\cdot \nabla)u + (D\sigma_k)^T u\big) \big\rangle  \\
	&= \sum_{k} \big\< (\sigma_k\cdot \nabla)u + (D\sigma_k)^T u, (\sigma_k\cdot \nabla)u + (D\sigma_k)^T u \big\>  \\
	&= \sum_{k} \langle \L_k u, \L_k u\rangle
	+ \sum_{k} \big\langle (D\sigma_k)^T u, (D\sigma_k)^T u\big\rangle
	+2 \sum_{k} \big\langle \L_k u, (D\sigma_k)^T u  \big\rangle.
\end{align*}

Recall  $Q(0)=cI_2$, note that $\|\nabla u\|_{L^2}=\|\nabla (K*\omega)\|_{L^2}=\|\omega\|_{L^2}$ and
\begin{align*}
  \sum_{k}\langle G*\omega, \L_k^2\omega\rangle &= \big\langle G*\omega,\div (Q(0)\nabla \omega) \big\rangle=-c\|\omega\|_{L^2}^2, \\
  \sum_{k} \langle \L_k u, \L_k u\rangle &= \langle \nabla u,Q(0)\nabla u  \rangle =c\|\nabla u\|_{L^2}^2,
\end{align*}
hence we get
\begin{align*}
	&\sum_{k}\langle G*\omega, \L_k^2\omega\rangle + \sum_{k}\langle G*\L_k\omega, \L_k\omega\rangle   \\
	&= \sum_{k} \big\langle (D\sigma_k)^T u, (D\sigma_k)^T u\big\rangle
	+2 \sum_{k} \big\langle \L_k u, (D\sigma_k)^T u  \big\rangle \\
	&=\int_{\mathbb{R}^2} \sum_{k}\partial_i\sigma_k^j(x)u^j(x) \partial_i\sigma_k^l(x) u^l(x) \,\d x
	+2\int_{\mathbb{R}^2} \sum_{k}\sigma_k^j(x)\partial_ju^i(x) \partial_i\sigma_k^l(x) u^l(x) \,\d x.
\end{align*}
Lemma \ref{reg-Q} implies that $\sum_{k}\partial_i\sigma_k^j(x)\partial_i\sigma_k^l(x)=-\Delta Q^{jl}(0)$ and $\sum_{k}\sigma_k^j(x)\partial_i\sigma_k^l(x)=-\partial_i Q^{jl}(0)=0$ since $Q$ is even by assumption. Hence we obtain

\begin{equation*}
	\sum_{k}\langle G*\omega, \L_k^2\omega\rangle + \sum_{k}\langle G*\L_k\omega, \L_k\omega\rangle
	\lesssim_Q \|u\|_{L^2}^2 \lesssim \|\omega\|_{\dot{H}^{-1}}^2.
\end{equation*}
Substituting it into \eqref{A-Ito-H^-1} and taking expectation, we get
\begin{equation}\label{append-omega}
	\frac{\d}{\d t} \E\big[\|\omega\|_{\dot{H}^{-1}}^2\big] \lesssim  \E\big[\|\omega\|_{\dot{H}^{-1}}^2 \big] +  \E\big[\|\theta\|_{L^2}^2\big].
\end{equation}
Combining the $L^2$ estimate \eqref{A-L^2-bdd} and above differential inequality, we get the bound
\begin{equation}
	\sup_{t\in [0,T]}\E\big[\|\omega_t\|_{\dot{H}^{-1}}^2\big]\leq C\big(T,\|\omega_0\|_{\dot{H}^{-1}},\|\theta_0\|_{L^2} \big).
\end{equation}

Now we estimate the $C\big([0,T],\dot{H}^{-1}\big)$ norm of $\omega$: by \eqref{A-Ito-H^-1} and the estimates above \eqref{append-omega},
\begin{equation}\label{append-omega-uniform-time}
\aligned
	\E\bigg[\sup_{t\in [0,T]} \|\omega_t\|_{\dot{H}^{-1}}^2 \bigg]
	&\leq \|\omega_0\|_{\dot{H}^{-1}}^2 + C\int_{0}^{T}\E\big[\|\omega_t\|_{\dot{H}^{-1}}^2 \big] +  \E\big[\|\theta_t\|_{L^2}^2\big]\, \d t \\
	&\quad +\E\bigg[\sup_{t\in [0,T]}\bigg|\sum_{k}\int_{0}^{t}\langle G*\omega_s, \L_k\omega_s  \rangle \,\d W_s^k\bigg|\bigg].
\endaligned
\end{equation}
To bound the martingale term, by the Burkholder-Davis-Gundy inequality, we obtain
\begin{align*}
	\E\bigg[\sup_{t\in [0,T]}\bigg|\sum_{k}\int_{0}^{t}\langle G*\omega_s, \L_k\omega_s  \rangle \,\d W_s^k\bigg|\bigg]
	\leq \E\Bigg[\bigg(\int_{0}^{T}\sum_{k}|\langle G*\omega_t, \L_k\omega_t \rangle|^2 \,\d t\bigg)^\frac{1}{2}\Bigg].
\end{align*}
Similarly as the arguments below \eqref{identity}, by \eqref{A-identity} and Young's inequality for convolution,
\begin{equation*}
	\begin{split}
		\sum_{k}|\langle G*\omega, \L_k\omega \rangle|^2
		&=\sum_{k}\Big|\Big\langle G*\omega, \curl \big((\sigma_k\cdot \nabla)u + (D\sigma_k)^T u\big) \Big\rangle\Big|^2 \\
		& =\sum_{k}\big|\big\langle u, (D\sigma_k)^T u \big\rangle \big|^2 \\
		&= \big\langle u^iu^j,\partial_{il}Q^{jm}* (u^lu^m) \big\rangle \\
		&\leq \|u\otimes u\|_{L^1}^2 \|D^2Q\|_{L^\infty}\lesssim_Q \|u\|_{L^2}^4.
	\end{split}	
\end{equation*}
Hence we have
\begin{align*}
	\E\bigg[\sup_{t\in [0,T]}\bigg|\sum_{k}\int_{0}^{t}\langle G*\omega_s, \L_k\omega_s  \rangle \,\d W^k\bigg|\bigg]
	&\lesssim \E\Bigg[ \bigg(\int_{0}^{T}\|u_t\|_{L^2}^4\, \d t\bigg)^{\frac{1}{2}} \Bigg] \\
	&\leq \E\Bigg[ \sup_{t\in [0,T]}\|u_t\|_{L^2}\bigg(\int_{0}^{T}\|u_t\|_{L^2}^2\, \d t\bigg)^{\frac{1}{2}} \Bigg]  \\
	&\leq \frac{1}{2} \E\bigg[\sup_{t\in [0,T]} \|\omega_t\|_{\dot{H}^{-1}}^2 \bigg] + C\E\bigg[\int_{0}^{T} \|\omega_t\|_{\dot{H}^{-1}}^2 \, \d t\bigg],
\end{align*}
which, combined with \eqref{append-omega-uniform-time}, yields the estimate
\begin{equation}\label{A-w-H^-1-bdd}
	\E\Big[ \|\omega\|_{C ([0,T],\dot{H}^{-1} )}^2 \Big]\leq C\big(T,\|\omega_0\|_{\dot{H}^{-1}},\|\theta_0\|_{L^2}\big) .
\end{equation}

As for the $\dot{H}^{-1}$ norm of $\theta$, we use the Ladyzhenskaya's inequality (see e.g., \cite[Section 1.1]{KukShi}),
\begin{equation*}
	\begin{split}
		\frac{\d}{\d t}\langle \theta,G*\theta \rangle + 2\|\theta\|_{L^2}^2
		&= 2\langle u\cdot\nabla G*\theta,\theta\rangle  \\
		&\leq 2\|u\|_{L^2}^{1/2}\|u\|_{\dot{H}^{1}}^{1/2}\|\nabla G*\theta\|_{L^2}^{1/2}\|\nabla G*\theta\|_{\dot{H}^{1}}^{1/2}\|\theta\|_{L^2} \\
		&\lesssim \|\omega\|_{\dot{H}^{-1}}^{1/2}\|\omega\|_{L^2}^{1/2} \|\theta\|_{\dot{H}^{-1}}^{1/2}\|\theta\|_{L^2}^{3/2} \\
	    &\leq\|\theta\|_{L^2}^{6}+ C\big(\|\omega\|_{\dot{H}^{-1}}^{2} + \|\omega\|_{L^2}^{2} + \|\theta\|_{\dot{H}^{-1}}^{2}\big).
	\end{split}
\end{equation*}
Combining the $L^2$ bound of $(\omega,\theta)$ and $\dot{H}^{-1}$ bound of $\omega$, we obtain the estimate
\begin{equation}\label{A-theta-H^-1-bdd}
	\E\Big[ \|\theta\|_{C ([0,T],\dot{H}^{-1} )}^2 \Big]\leq C\big(T,\|\omega_0\|_{\dot{H}^{-1}}, \|\theta_0\|_{\dot{H}^{-1} \cap L^2} \big) .
\end{equation}
The bounds \eqref{A-L^2-bdd}, \eqref{A-w-H^-1-bdd} and \eqref{A-theta-H^-1-bdd} are the desired estimates.

\subsection*{Step 2: Higher order estimates}
Now we are going to derive upper bounds for higher-order derivatives. Applying the It\^o's formula again, we have, for all $m\in \N_+$,

\begin{equation}\label{A-Ito-High}
	\begin{split}
		\d\, \langle D^m\omega, D^m\omega\rangle
		&=-2 \big\langle D^m\omega, D^m(u\cdot \nabla \omega) \big\rangle \,\d t + 2\langle D^m\omega, D^m\partial_1\theta\rangle\,\d t + 2 \langle D^m\omega, \Delta D^m\omega \rangle \,\d t\\
		&\quad - 2\sum_{k} \langle D^m\omega,  D^m\L_k\omega\rangle \,\d W^k \\
		&\quad + \sum_{k} \langle  D^m\L_k^2\omega, D^m\omega\rangle \,\d t + \sum_{k} \langle  D^m\L_k\omega, D^m\L_k\omega\rangle \,\d t
	\end{split}
\end{equation}
and
\begin{equation*}
	\begin{split}
		\frac{\d}{\d t}\, \langle D^{m}\theta, D^{m}\theta\rangle
		&=-2 \big\langle D^{m}\theta, D^{m}(u\cdot \nabla \theta) \big\rangle  + 2 \langle D^{m}\theta, \Delta D^{m}\theta \rangle.
	\end{split}
\end{equation*}
Integrating by parts yields
\begin{equation}\label{A-laplace}
	\langle D^m\omega, \Delta D^m\omega \rangle=-\|D^{m+1}\omega\|_{L^2}^2
	,\quad \langle D^m\theta, \Delta D^m\theta \rangle=-\|D^{m+1}\theta\|_{L^2}^2.
\end{equation}
The control of nonlinear parts by the Laplacian term is classical, we proceed as in \cite[Section 2.1]{KukShi}. Since $u\cdot \nabla\omega=\curl (u\cdot \nabla u)$ and $\big\langle D^{m+1}u, u\cdot \nabla( D^{m+1}  u) \big\rangle=0$ due to $\div u=0$, we have
\begin{equation*}
	\begin{split}
		\big\langle D^m\omega, D^m(u\cdot \nabla \omega) \big\rangle
		&=\big\langle D^m\curl u, D^m\curl (u\cdot \nabla u) \big\rangle \\
		&=\big\langle D^{m+1} u, D^{m+1} (u\cdot \nabla u) \big\rangle \\
		&= \sum_{n<m+1}C_{m,n}\big\langle D^{m+1}u, D^{m+1-n} u\cdot \nabla D^n u \big\rangle.
	\end{split}
\end{equation*}
By Ladyzhenskaya's inequality, we have
\begin{align*}
	\big|\big\langle D^{m+1}u, D^{m+1-n} u\cdot \nabla D^n u \big\rangle\big|\leq C  \|u\|_{\dot{H}^{m+3/2}}\|u\|_{\dot{H}^{m+3/2-n}}\|u\|_{\dot{H}^{n+1}}
\end{align*}
Since $0\leq n<m+1$, the numbers $m+3/2, m+3/2-n, n+1$ lie between $1$ and $m+2$. Then by interpolation inequality between $\dot{H}^{1}$ and $\dot{H}^{m+2}$, we could get the bound
\begin{equation}\label{A-nonlinear-w}
	\big| \big\langle D^m\omega, D^m(u\cdot \nabla \omega) \big\rangle \big| \lesssim \|u\|_{\dot{H}^{1}}^{\frac{m+2}{m+1}} \|u\|_{\dot{H}^{m+2}}^{\frac{2m+1}{m+1}}
	\leq \frac{1}{2}\|D^{m+1}\omega\|_{L^2}^2 + C\|\omega\|_{L^2}^{2m+2}.
\end{equation}
The estimate of $\big\langle D^m\theta, D^m(u\cdot \nabla \theta) \big\rangle$ can be done in a rather similar manner, while we interpolate between $\|u\|_{\dot{H}^1}$ (resp. $\|\theta\|_{L^2}$) and $\|u\|_{\dot{H}^{m+1}}$  (resp. $\|\theta\|_{\dot{H}^{m+1}}$), finally arrive at
\begin{equation}\label{A-nonlinear-theta}
	\big|\big\langle D^m\theta, D^m(u\cdot \nabla \theta) \big\rangle \big|\leq \frac{1}{2} \|D^{m+1}\theta\|_{L^2}^2 + \|D^{m}\omega\|_{L^2}^2 +C\big(T,m,\|\omega_0\|_{L^2}, \|\omega_0\|_{L^2}\big) .
\end{equation}
Combining the estimates \eqref{A-L^2-bdd}, \eqref{A-laplace}, \eqref{A-nonlinear-w} and \eqref{A-nonlinear-theta}, we bound the nonlinear terms by the viscous and diffusion terms.

Our task now is to show that the It\^o-Stratonovich correction terms, i.e, the last two terms in \eqref{A-Ito-High} do not blow up. The key point is that higher derivatives vanish due to a commutator estimate from \cite[Section 4.4]{CFH}, see also \cite[Appendix]{LanCri}. More specifically, define two commutators as $S_k = [D^m, \L_k]$ and $T_k= [S_k, \L_k]$ for $k\in \N_+$. It is easy to check that $S_k$ and $T_k$ are both differential operators with orders not more than $m$ and the coefficients are the derivatives of $\sigma_k$'s. Though surprising, the following identity holds.
\begin{lem}\label{A-noise-high}
	For all $k,m\in \N_+$, we have
	\begin{equation*}
		\langle  D^m\L_k^2\omega, D^m\omega\rangle + \langle  D^m\L_k\omega, D^m\L_k\omega\rangle = \langle T_k\omega, D^m\omega \rangle + \langle S_k\omega,S_k\omega\rangle.
	\end{equation*}
\end{lem}
\begin{proof}
	Since $\div\sigma_k=0$, we know that $\L_k^*$, the dual of $\L_k$, is equal to $-\L_k$; hence,
	\begin{equation*}
		\begin{split}
	    	\langle  D^m\L_k^2\omega, D^m\omega\rangle
	    	 &= \langle  \overbrace{(D^m\L_k-\L_kD^m)}^{S_k}\L_k\omega, D^m\omega\rangle
	    	 +\langle \L_kD^m\L_k\omega, D^m\omega\rangle \\
	    	 &= \langle S_k\L_k\omega, D^m\omega \rangle - \langle D^m\L_k\omega, \L_kD^m\omega \rangle \\
	    	 &=\langle S_k\L_k\omega, D^m\omega \rangle - \langle D^m\L_k\omega, D^m\L_k\omega \rangle + \langle D^m\L_k\omega, S_k\omega  \rangle.
		\end{split}
	\end{equation*}
As a result,
    \begin{align*}
    	&\langle  D^m\L_k^2\omega, D^m\omega\rangle + \langle  D^m\L_k\omega, D^m\L_k\omega\rangle  \\
    	&= \langle S_k\L_k\omega, D^m\omega \rangle + \langle D^m\L_k\omega, S_k\omega  \rangle  \\
    	&= \langle S_k\L_k\omega, D^m\omega \rangle + \langle S_k\omega,S_k\omega \rangle + \langle \L_kD^m\omega, S_k\omega\rangle \\
    	&=\langle T_k\omega, D^m\omega \rangle + \langle S_k\omega,S_k\omega\rangle.\qedhere
    \end{align*}
\end{proof}
From the above lemma, and the assumption that derivatives of all orders of $\sigma_k$'s are summable, we obtain the bound
\begin{equation}\label{A-Ito-Str-bdd}
	\sum_{k} \langle  D^m\L_k^2\omega, D^m\omega\rangle + \langle  D^m\L_k\omega, D^m\L_k\omega\rangle \lesssim_{Q,m} \|D^m\omega\|_{L^2}^2 + \|\omega\|_{L^2}^2.
\end{equation}

Similarly to \textbf{Step 1}, we can bound the expectation of the martingale term in \eqref{A-Ito-High}, the details are omitted. In conclusion, combining all the estimates \eqref{A-L^2-bdd}, \eqref{A-w-H^-1-bdd}, \eqref{A-theta-H^-1-bdd} and \eqref{A-Ito-Str-bdd}, we arrive at
\begin{equation*}
	\E\Big[ \|\omega\|_{C ([0,T],\dot{H}^{-1}\cap \dot{H}^{m} )}^2 \Big]
	+ \E\Big[ \|\theta\|_{C ([0,T],\dot{H}^{-1}\cap \dot{H}^{m} )}^2 \Big]
	\leq C\big(T,m,Q, \|\omega_0\|_{\dot{H}^{-1}\cap \dot{H}^{m}},\|\theta_0\|_{\dot{H}^{-1}\cap \dot{H}^{m}} \big),
\end{equation*}
which is sufficient to implement the classical compactness methods to prove Lemma \ref{wellpoesd-reg-model}.

\begin{rmk}
	The term $\Delta \omega$ in \eqref{regular-model-A} is only used to bound the nonlinear term $u\cdot\nabla\omega$ in the estimates of higher derivatives. However, the global smooth solution to Boussinesq system without $\Delta\omega$ also exist, see \cite{Cha} for deterministic case. Since smooth transport noise does not change higher order energy estimate (Lemma \ref{A-noise-high}), we expect that the results of \cite{Cha} can be extended to stochastic case, though, as far as we know, there is no explicit statements of such a theorem in the literature. In fact, this has been achieved for the 2D stochastic Euler equation with smooth transport noise, see \cite{LanCri}. So we choose to construct smooth approximate sequence by considering the viscous Boussinesq system just for convenience.
\end{rmk}

\section{Proof of Corollary \ref{cor3}}
\begin{proof}
It suffices to prove that $\{\theta^{\delta_n}\}$ is a Cauchy sequence in $\theta\in{C_TL^2}\cap L_T^2\dot{H}^1\cap L_T^1(\dot{W}^{1,1}\cap \dot{W}^{1,p})$, since the limit in this space must coincide with $\theta$. To simplify notations, for $m,n\in \N$, let $\theta^1=\theta^{\delta_n},\theta^2=\theta^{\delta_m},u^1=u^{\delta_n},u^2=u^{\delta_m}$ and $\bar{\theta}=\theta^1-\theta^2,\bar{u}=u^1-u^2$. The difference $\bar{\theta}$ satisfies the equation
\begin{equation}\label{bar-theta}
	\partial_t\bar{\theta} + u^{1}\cdot \nabla\bar{\theta} + \bar{u}\cdot \nabla\theta^{2} = \Delta \bar{\theta}.
\end{equation}

First we consider $L^2$ estimate, by H\"older's inequality,
\begin{align*}
	\frac{1}{2}\frac{\d}{\d t}\|\bar{\theta}\|_{L^2}^2
	+\|\nabla \bar{\theta}\|_{L^2}^2
	=\<\bar{u}\cdot\nabla\bar{\theta},\theta^2\>
	\leq \|\nabla \bar{\theta}\|_{L^2}
	\|\bar{u}\|_{L^q}\|\theta^2\|_{L^r}
	\leq \frac{1}{2}\|\nabla \bar{\theta}\|_{L^2}^2
	+\frac{1}{2}\|\bar{u}\|_{L^q}^2\|\theta^2\|_{L^r}^2,
\end{align*}
where $q,r\in (2,\infty)$ with $q^{-1}+r^{-1}=2^{-1}$. Integrating  in time, we have $\P\text{-a.s.}$,
\begin{equation*}
	\|\bar{\theta}_t\|_{L^2}^2
	+ \int_{0}^{t}\|\nabla \bar{\theta}_s\|_{L^2}^2\,\d s
	\leq \|\bar{u}\|_{L_T^{\infty}L^q}^2
	\int_{0}^{t} \|\theta_s^2\|_{L^r}^2\,\d s, \quad \forall t\in [0,T].
\end{equation*}
By Sobolev embedding and interpolation, we have $\P\text{-a.s.}$,
\begin{align*}
	\int_{0}^{T}\|\theta_s^2\|_{L^r}^2\,\d s
	\leq \int_{0}^{T}\|\theta_s^2\|_{L^2}^{4/r}\|\nabla\theta_s^2\|_{L^2}^{2-4/r}\,\d s
	\leq \int_{0}^{T}\|\theta_s^2\|_{L^2}^{2}+\|\nabla\theta_s^2\|_{L^2}^{2}\,\d s
	\lesssim_T \|\theta_0^2\|_{L^2}^2.
\end{align*}
Note that if  $r^{-1}+(p\wedge 2)^{-1}<1$, or equivalently $2<q<\frac{2(p\wedge 2)}{2-p\wedge 2}$, then by Corollary \ref{cor2}, we have $\|\bar{u}\|_{L_T^\infty L^q}=\|u^{\delta_n}-u^{\delta_m}\|_{L_T^\infty L^q}\rightarrow 0$ as $m,n\rightarrow \infty,~\P\text{-a.s.}$ Thus we have obtained the convergence
\begin{equation*}
	\|\bar{\theta}\|_{C_TL^2}^2 + \|\nabla\bar{\theta}\|_{L_T^2L^2}^2\rightarrow 0,\quad \text{ as }m,n\rightarrow \infty,~\P\text{-a.s.}
\end{equation*}

It remains to control $\|\nabla\bar{\theta}\|_{L_T^1(L^1\cap L^p)}$. We also write \eqref{bar-theta} in mild form:
\begin{equation}\label{bar-theta-Lp}
	\bar{\theta}_t=e^{t\Delta}\bar{\theta}_0 -\int_{0}^{t} e^{(t-s)\Delta} \big(u_s^1\cdot\nabla\bar{\theta}_s\big)\,\d s
	-\int_{0}^{t} e^{(t-s)\Delta} \big(\bar{u}_s\cdot\nabla\theta_s^2\big)\,\d s,\quad \forall t\in [0,T].
\end{equation}

From  $\P\text{-a.s.}$ convergence $\|\bar{u}\|_{L_T^\infty L^q}+\|\nabla\bar{\theta}\|_{L_T^2L^2}^2\rightarrow 0$ as $m,n\rightarrow\infty$ for $2<q<\frac{2(p\wedge 2)}{2-p\wedge 2}$ and the convergence of initial data,
we can estimate the above three terms in \eqref{bar-theta-Lp} as \textbf{Cases 2}--\textbf{3} in Section \ref{appr} respectively, and  consequently obtain
\begin{align*}
	\int_{0}^{T} \|\nabla\bar{\theta}_t\|_{L^1\cap L^p}\,\d t
	\rightarrow 0,\quad \text{ as }m,n\rightarrow\infty,~ \P\text{-a.s.}
\end{align*}
The proof is complete now.
\end{proof}

\medskip

\noindent \textbf{Acknowledgements.}
We thank Prof. Guohuan Zhao for helpful discussions on the heat kernel with singular drifts.
The second author is grateful to the financial supports of the National Key R\&D Program of China (No. 2020YFA0712700), the National Natural Science Foundation of China (12090010, 12090014).

\end{document}